\documentclass[leqno]{amsart}
\usepackage{amsmath,amsfonts,amsthm,amssymb,indentfirst,epic,url}

\newcommand{\<}{\kern.0833em}
\newcommand{\Gp}{{\mathbf{Group}}}
\newcommand{\Rg}{\mathbf{Ring}^1}
\renewcommand{\r}{\mathrm}
\newcommand{\cP}{\raisebox{.1em}{$\<\scriptscriptstyle\coprod$}\nolinebreak[2]}
\newcommand{\V}{\fb{V}}
\newcommand{\strt}[1][1.7]{\vrule width0pt height0pt depth#1pt}

\newtheorem{theorem}{Theorem}
\newtheorem{lemma}[theorem]{Lemma}
\newtheorem{corollary}[theorem]{Corollary}
\newtheorem{proposition}[theorem]{Proposition}
\newtheorem{definition}[theorem]{Definition}
\newtheorem{question}[theorem]{Question}

\newcommand{\fb}{\mathbf}

\begin{document}

\title[Inner endomorphisms]{An inner automorphism is only an
inner automorphism,\\
but an inner endomorphism can be something strange}%
\thanks{%
ArXiv URL: \ \url{http://arXiv.org/abs/1001.1391}\,.
\\
\strt\hspace{1.5em}I discovered
the main results of Theorems~\ref{T.in_aut_G} and~\ref{T.end_R-Rg}
several decades ago, at which time I was partly supported
by a National Science Foundation grant, but
I cannot now reconstruct the date, and hence the grant number.
I spoke on those results at the January, 2009 AMS-MAA
Joint Meeting in Washington D.C..
After publication of this note, updates, errata,
related references etc., if found, will be recorded at
\url{http://math.berkeley.edu/~gbergman/papers/}
}


\author{George M. Bergman}
\address{George M. Bergman\\
Department of Mathematics\\
University of California\\
Berkeley, CA 94720-3840\\
USA}
\email{gbergman@math.berkeley.edu}

\begin{abstract}
The inner automorphisms of a group $G$ can be characterized
within the category of groups
without reference to group elements: they are precisely those
automorphisms of $G$ that can be extended, in a functorial
manner, to all groups $H$ given with homomorphisms $G\to H.$
(Precise statement in \S\ref{S.group}.)
The group of such extended systems of automorphisms,
unlike the group of inner automorphisms of $G$ itself,
is always isomorphic to $G.$
A similar characterization holds for inner automorphisms of an
associative algebra $R$ over a field $K;$
here the group of functorial systems of automorphisms is isomorphic
to the group of units of $R$ modulo the units of $K.$

If one looks at the above functorial extendibility property for
endomorphisms, rather than just automorphisms,
then in the group case, the only additional example
is the trivial endomorphism; but
in the $\!K\!$-algebra case, a construction unfamiliar to
ring theorists, but known to functional analysts, also arises.

Systems of endomorphisms with the same
functoriality property are examined in some other categories;
other uses of the phrase ``inner endomorphism'' in the literature,
some overlapping the one introduced here, are noted; the
concept of an inner {\em derivation} of an associative or
Lie algebra is looked at from the same point of view, and the dual
concept of a ``co-inner'' endomorphism is briefly examined.
Several open questions are noted.
\end{abstract}
\maketitle

\section*{Overview.}

You can read this overview if you'd like to know the topics of the
various sections to come; but
feel free to skip it if you'd prefer to plunge in, and let the
story tell itself.

In \S\ref{S.group}, we motivate the approach of this paper using the
case of groups.
We obtain the characterization of inner automorphisms of groups
that is stated in the abstract, and, modeled on this, we define
concepts of inner automorphism and inner endomorphism for an
object of a general category.

In \S\ref{S.ring}, these definitions are applied to associative unital
algebras over a commutative ring $K,$ and full characterizations
of the inner automorphisms and endomorphisms are
obtained in the case where $K$ is a field.
In \S\ref{S.non-field}, counterexamples are given to the obvious
generalizations of these results to base rings that are
not fields, and the question of what the general
inner automorphism and endomorphism might look like
in that case is examined.

In \S\ref{S.literature} we pause to survey concepts that have
appeared in the literature under the name ``inner endomorphism'',
with varying degrees of overlap with that of this note.

\S\ref{S.other} contains some easy observations on
inner automorphisms and endomorphisms (in the sense here defined)
on a few other sorts of algebraic objects.

In \S\S\ref{S.d_assoc}-\ref{S.Lie} we study inner {\em derivations}
of associative and Lie algebras, and also inner endomorphisms of
Lie algebras, pausing in \S\ref{S.deriv_def} to
consider what the general definition of ``inner derivation'' should be.

It is noted in \S\ref{S.co-inner} that our concept of inner
endomorphism dualizes to one of
{\em co-inner endomorphism}, and we determine these for the
category of $\!G\!$-sets, for $G$ a group.

In \S\ref{S.concl} we briefly look at the ideas of this paper from
the perspective of the theory of representable functors.

Although we were not able to obtain in \S\ref{S.non-field} a
full description of the inner endomorphisms of an
associative $\!K\!$-algebra when $K$ is not a field,
we prove in a final appendix, \S\ref{S.1-1}, using the partial
results of \S\ref{S.non-field}, that such inner endomorphisms
are always one-to-one.

Open questions are noted in \S\S\ref{S.group},
\ref{S.non-field},
\ref{S.d_assoc} and~\ref{S.Lie}.

I am indebted to Bill Arveson for helpful references,
and to the referee for many useful suggestions.

\section{Inner automorphisms and inner endomorphisms of groups.}\label{S.group}
Recall that an automorphism $\alpha$ of a group $G$ is called
{\em inner} if there exists an $s\in G$
such that $\alpha$ is given by conjugation by $s:$
\begin{equation}\label{d.cj_g}
\alpha(t)\ =\ s\,t\,s^{-1}\quad(t\in G).
\end{equation}

Given this definition, it may seem perverse to ask whether the
condition that $\alpha$ be inner can be characterized
without speaking of group elements.
Note, however, that the definition implies
the following property, which can indeed be so stated:
For every homomorphism $f$ of $G$ into a group $H,$ there
exists an automorphism $\beta_f$ of $H$
making a commuting square with $\alpha:$
\begin{equation}\begin{minipage}[c]{5pc}\label{d.*b}
\begin{picture}(50,55)
\multiput(0,5)(0,34){2}{
	\put(0,0){$G$}\put(13,3){\vector(1,0){20}}\put(34,0){$H$}
	\put(18,8){$f$}}
\put(45,5){.}
\put(6,35){\vector(0,-1){20}}\put(7,25){$\alpha$}
\put(40,35){\vector(0,-1){20}}\put(41,25){$\beta_f$}
\end{picture}\end{minipage}\end{equation}
(Namely, we can take $\beta_f$ to be conjugation by $f(s).)$

Whether this property alone is equivalent to $\alpha$ being inner,
I do not know; but the above conclusion can be strengthened.
Let $\alpha$ be as in~(\ref{d.cj_g}), and for every
group $H$ and homomorphism $f: G\to H$ let $\beta_f$ be, as
above, the inner automorphism of $H$ induced by $f(s).$
Then this system of automorphisms is ``coherent'',
in that for every commuting triangle of group homomorphisms
\begin{equation}\begin{minipage}[c]{5pc}\label{d.triangle}
\begin{picture}(50,50)
\put(0,22){$G$}
\put(14,31){\vector(2,1){22}}
	\put(18,41){$f_1$}
\put(14,19){\vector(2,-1){22}}
	\put(18,4){$f_2$}
\put(38,01){$H_2$}
\put(38,42){$H_1$}
\put(43,38){\vector(0,-1){26}}
	\put(44,25){$h$}
\end{picture}
\end{minipage}\end{equation}
one has
\begin{equation}\label{d.funct}
\beta_{f_2}\,h\ =\ h\,\beta_{f_1}.
\end{equation}

Let us show that this strengthened statement {\em is} equivalent to
$\alpha$ being inner; and that the family of morphisms $\beta_f$ does
what $\alpha$ alone in general does not: it uniquely determines $s.$

\begin{theorem}\label{T.in_aut_G}
Let $G$ be a group and $\alpha$ an automorphism of $G.$
Suppose we are given, for each group $H$ and homomorphism
\begin{equation}\label{d.f}
f:G\to H,
\end{equation}
an automorphism $\beta_f$ of $H,$ with the properties that\\[.2em]
\textup{(i)}\ \ $\beta_{\r{id}_G}\ =\ \alpha,$ and\\[.2em]
\textup{(ii)}\ \ for every commuting
triangle~\textup{(\ref{d.triangle})}
one has\textup{~(\ref{d.funct})}.\vspace{.2em}

Then there is a unique $s\in G$ such that for all $H$ and $f$
as in\textup{~(\ref{d.f})}, one has
\begin{equation}\label{d.*b_f}
\beta_f(t)\ =\ f(s)\,t\,f(s)^{-1}\quad (t\in H).
\end{equation}

In particular, $\alpha$ is inner.
Thus, an automorphism $\alpha$ of a group $G$ is inner if and only
if there exists such a system of automorphisms $\beta_f.$
\end{theorem}

\begin{proof}
To investigate the system of maps $\beta_f,$ let us look
at their behavior on a ``generic'' element.
Since the domains of these maps are groups with homomorphisms of $G$
into them, a {\em generic} member of such a group will be the
element $x$ of the group $G\langle x\rangle$ obtained by adjoining
to $G$ one additional element $x$ and no additional relations.
(This group is the coproduct $G\cP\langle x\rangle$ of $G$ with the
free group $\langle x\rangle$ on one generator; in group-theorists'
language and notation, the free product $G*\langle x\rangle.)$

So letting $\eta$ be the inclusion map $G\to G\langle x\rangle,$
consider the element $\beta_\eta(x)\in G\langle x\rangle.$
By the structure theorem for coproducts of groups,
this can be written $w(x),$ where $w$ is a reduced word in $x$
and the elements of $G.$
Note that for any map $f$ of $G$ into a group $H,$ and any element
$t\in H,$ we can form a triangle~(\ref{d.triangle}) with
$H_1=G\langle x\rangle,$ $f_1=\eta,$ $H_2=H,$ $f_2=f,$ and $h(x)=t.$
(There is a unique such $h$ making~(\ref{d.triangle}) commute,
by the universal property of $G\langle x\rangle.)$
By~(\ref{d.funct}), the element
$\beta_f\,h(x)=\beta_f(t)$ is equal to
$h\,\beta_\eta(x)=h\,w(x)=w_f(t),$ where $w_f$ denotes the
result of substituting for the elements of $G$ in the
word $w$ their images under $f.$
Thus, $\beta_f$ acts by carrying every $t\in H$ to $w_f(t).$

Conversely, starting with any element $w(x)\in G\langle x\rangle,$
the formula $\beta_f(t)=w_f(t)$ clearly gives a {\em set}~map
$\beta_f: H\to H$ for each $f$ as in~(\ref{d.f}), in such a way
that~(ii) above holds.
To determine when these set maps respect the group operation, we should
consider the effect of the map induced by
$w(x)$ on the {\em product} of a generic {\em pair} of elements.
So we now take the group $G\langle x_0,x_1\rangle$ gotten by
adjoining to $G$ two elements and no relations, let $\eta$ be the
inclusion of $G$ therein, and consider the relation
$\beta_\eta(x_0\,x_1)\ =\ \beta_\eta(x_0)\,\beta_\eta(x_1),$ i.e.,
\begin{equation}\label{d.wxy}
w(x_0\,x_1)\ =\ w(x_0)\,w(x_1).
\end{equation}

When we transform the product on
the right-hand side of~(\ref{d.wxy}) into a reduced
word in $x_0,\ x_1$ and nonidentity elements of $G,$ the only
reduction that can occur is the simplification of the product
of the factors from $G$ at the right end of $w(x_0)$ and
at the left end of $w(x_1);$ in particular,
all occurrences of $x_0$ continue to occur
to the left of all occurrences of $x_1.$
On the other hand, in the left side of~(\ref{d.wxy}), each
occurrence of $x_0$ or $x_1$ is adjacent to an occurrence of the other.
It is easy to deduce that $w(x)$ can contain at most one occurrence
of $x,$ and that such an $x,$ if it occurs, must have exponent $+1.$
Moreover, if there were no occurrences of $x,$ then the functions
$\beta_f$ would be constant, hence could not give automorphisms
of nontrivial groups $H;$ so $w(x)$ must have the
form $s_0\,x\,s_1.$
Substituting into~(\ref{d.wxy}), we find that $s_1\,s_0=1;$
hence letting $s=s_0,$ we have $w(x)= s\,x\,s^{-1}.$

Thus, the maps $\beta_f$ have the form~(\ref{d.*b_f}).
Moreover, distinct elements $s$ give distinct words $w(x),$ hence
give distinct systems of automorphisms, since they act differently
on $x\in G\langle x\rangle;$ so such a system of automorphisms
determines $s$ uniquely.

Combining the above description of $\beta_f$ with condition~(i),
we see that our original automorphism $\alpha$ is inner.
This gives the ``if\,'' direction of the final sentence of the theorem;
the remarks at the beginning of this section give ``only if\,''.
\end{proof}

(In the above theorem, we did not explicitly
assume commutativity of the diagrams~(\ref{d.*b}).
But in view of condition~(i), that commutativity requirement,
for given $h,$ is the case of condition~(ii) where
$H_1=G,$ $f_1=\r{id}_G,$ $H_2=H,$ $h=f_2=f.)$

For fixed $G$ one can clearly compose two coherent systems of
automorphisms of the sort considered in Theorem~\ref{T.in_aut_G}
to get another such system; and we see from the theorem that under
composition, these systems form a group isomorphic to $G.$

There is
an elegant formulation of this fact in terms of comma categories.
Recall that for any object $X$ of a category $\fb{C},$
the category whose objects are objects $Y$ of $\fb{C}$ given with
morphisms $X\to Y,$ and whose morphisms are commuting triangles
analogous to~(\ref{d.triangle}), is denoted $(X\downarrow\fb{C})$
(called a ``comma category'' because of the older notation $(X,\fb{C});$
see~\cite[\S II.6]{CW}).
A system of maps $\beta_f$ as in Theorem~\ref{T.in_aut_G}
associates to each object $f:G\to H$ of the comma
category $(G\downarrow \Gp)$ an automorphism, not of that object,
but of the group $H;$ i.e., of the value, at that
object $f:G\to H,$ of the forgetful functor $(G\downarrow \Gp)\to\Gp.$
Our condition~(\ref{d.funct}) says that these automorphisms
$\beta_f$ should together comprise an automorphism of that forgetful
functor.
In summary,

\begin{theorem}\label{T.aut_forget}
For any group $G,$ the automorphism group of the forgetful functor
$U:(G\downarrow \Gp)\to\Gp$ is isomorphic to $G,$ via the map
taking each $s\in G$ to the automorphism of $U$
given by~\textup{(\ref{d.*b_f})}.\qed
\end{theorem}

In the proof of Theorem~\ref{T.in_aut_G} we used
the assumption that $\alpha$ and the $\beta_f$ were
automorphisms, rather than simply endomorphisms, only once;
to exclude the case where $w(x)$ contained no occurrences of $x.$
In that case,~(\ref{d.wxy}) forces $w$ to be the identity element,
whence the $\beta_f$ are the trivial endomorphisms, $\varepsilon(t)=1.$
So we have

\begin{corollary}[to proof of Theorem~\ref{T.in_aut_G}]\label{C.in_end_G}
If in the hypotheses of Theorem~\ref{T.in_aut_G} one everywhere
substitutes ``endomorphism'' for ``automorphism'', then the
possibilities for $(\beta_f)$
are as stated there, together with one additional
case: where every $\beta_f$ is the trivial endomorphism of $H.$
In the language of Theorem~\ref{T.aut_forget},
the endomorphism monoid of the forgetful functor
$(G\downarrow \Gp)\to\Gp$ is isomorphic to $G\cup\{\varepsilon\},$
where $\varepsilon$ is a zero element.
\qed
\end{corollary}

(There is nothing exotic about the trivial endomorphism;
so the second half of the title of this note does not apply to
the category of groups.)

Let us abstract, and name, the concepts we have been using.

\begin{definition}\label{D.inner}
If $X$ is an object of a category $\fb{C},$ then an endomorphism
\textup{(}respectively automorphism\textup{)}
of the forgetful functor $(X\downarrow\fb{C})\to\fb{C}$ will be called
an {\em extended} \textup{(}or if there is danger of
ambiguity, ``$\fb{C}\!$-extended''\textup{)} inner endomorphism
\textup{(}resp., inner automorphism\textup{)} of $X.$

An extended inner endomorphism or automorphism will at times
be denoted $(\beta_f),$ where $f$ is understood to run over
all $f:X\to Y$ in $(X\downarrow\fb{C}),$ and the $\beta_f$
are the corresponding endomorphisms or automorphisms of the
objects $Y.$

An endomorphism or automorphism of $X$ will be called {\em inner}
if it is the value at $\r{id}_X$ of an extended
inner endomorphism or automorphism of $X.$
When there is danger of ambiguity, one may add ``in the
category-theoretic sense'' and/or ``with respect to $\fb{C}\!$''.
\end{definition}

So, like ``monomorphism'' and ``epimorphism'',
the term ``inner'' will acquire a certain tension between
a pre-existing sense and a category-theoretic sense, which
will agree in many, but not necessarily in all cases where
the former is defined.
In subsequent sections we will study the
category-theoretic concept in several other categories.

We remark that if $\fb{C}$ is a legitimate category (i.e., if its
hom-sets $\fb{C}(X,Y)$ are small sets -- in classical
language, sets rather than proper classes -- but if its
object-set may be large), then the monoids of endomorphisms,
respectively, the groups of automorphisms, of the forgetful functors
$(X\downarrow\fb{C})\to\fb{C}$ are not, in general, {\em small} monoids
or groups.
However, there is a set-theoretic approach that handles such
size-problems
elegantly; see \cite[{\S}6.4]{245} (cf.\ \cite[{\S\S}I.6-7]{CW}).
Aside from this point, these constructions behave very
nicely:  if $f:X_1\to X_2$ is a morphism,
it is easy to see that an extended inner endomorphism or automorphism
of $X_1$ induces via $f$ an extended inner endomorphism or automorphism
of $X_2$ (in contrast with the behavior of ordinary endomorphisms,
automorphisms, and inner automorphisms); thus,
these constructions give {\em functors} from $\fb{C}$
to the categories of (possibly large) monoids and groups.
(However, we shall not use this observation below.)

We end our consideration of these concepts
in $\Gp$ by recording a question mentioned above,
generalized from automorphisms to endomorphisms.

\begin{question}\label{Q.group}
If $\alpha$ is an endomorphism of a group $G,$ such that for
each object $f:G\to H$ of $(G\downarrow\nolinebreak[3]\Gp)$
there exists an endomorphism $\beta_f$ of $H$ making the
diagram\textup{~(\ref{d.*b})} commute, must $\alpha$ then be inner
in the sense of Definition~\ref{D.inner}; i.e., is it then possible
to choose such endomorphisms $\beta_f$ so as to
satisfy\textup{~(\ref{d.funct})} for
all commuting triangles\textup{~(\ref{d.triangle})}?
\textup{(}By the preceding results, this is equivalent to:
Must $\alpha$ either be an inner automorphism in the classical sense,
or the trivial endomorphism?\textup{)}
\end{question}

\section{The case of $\!K\!$-algebras.}\label{S.ring}

Let us now consider the same  questions for rings.

Let $\Rg$ denote the category of all associative unital rings.
A general difficulty in the study of universal constructions in
this category is the nontriviality of the multilinear algebra of
abelian groups, i.e., $\!\mathbb{Z}\!$-modules.
Often things are no worse if we generalize our considerations
to the category $\Rg_K$ of associative unital algebras over
a general commutative ring $K,$ and they then become much better if we
assume $K$ a field.
Below, we shall begin the analysis of inner endomorphisms
of $\!K\!$-algebras for $K$ a general commutative ring;
then, about half-way through, we will have to
restrict ourselves to the case where $K$ is a field.
In the next section we will examine
what versions of our result might be true for general $K.$

So let $K$ be any commutative ring (where ``associative unital''
is understood), and $R$ any nonzero object of $\Rg_K.$

We will again use generic elements.
The extension of $R$ by a single generic element $x$
in $\Rg_K$ has the $\!K\!$-module decomposition
\begin{equation}\label{d.Rx}
R\langle x\rangle\ =\ \ R\ \oplus\ (R\,x R)\ \oplus
\ (R\,x R\,x R)\ \oplus\ \dots\ \ \cong
\ \ R\ \oplus\ (R\otimes R)\ \oplus\ (R\otimes R\otimes R)
\ \oplus\ \dots\,.
\end{equation}
Here the tensor products are as $\!K\!$-modules.
Tensor products over $K$ will be almost the only tensor products used
in this note, so we make the convention that $\otimes,$ without
a subscript, denotes $\otimes_K.$

The extension of $R$ by two generic elements similarly has form
\begin{equation}\label{d.Rx0x1}
R\langle x_0,\,x_1\rangle\ =
\bigoplus_{\substack{n\geq 0\\[.2em] i_1,\,\dots,\,i_n\in\{0,1\}}}
R\,x_{i_1}R\dots R\,x_{i_n}R\quad\cong
\bigoplus_{\substack{n\geq 0\\[.2em] i_1,\dots,i_n\in\{0,1\}}}
R\otimes R\otimes\ldots\otimes R\otimes R\,.
\end{equation}

Exactly as in the proof of Theorem~\ref{T.in_aut_G},
every extended inner endomorphism of $R$ will be determined
by the image under it of $x\in R\langle x\rangle,$ which will be some
element $w(x)\in R\langle x\rangle.$
And again, every $w(x)\in R\langle x\rangle$ induces,
for each object $f:R\to S$ of $(R\downarrow\Rg_K),$
a set map of $S$ into itself, sending each $r\in S$ to $w_f(r)\in S,$
and these maps respect morphisms among such objects.
So again, our task is to determine for which $w(x)\in R\langle x\rangle$
the induced set-maps $S\to S$ are $\!K\!$-algebra homomorphisms.

These maps will respect addition if and only if the required
equation holds in the generic case, i.e., if and only if,
in $R\,\langle x_0,\,x_1\rangle,$
\begin{equation}\label{d.wx0+x1}
w(x_0+x_1)\ =\ w(x_0)\ +\ w(x_1)\,.
\end{equation}
I claim that the only elements $w(x)\in R\langle x\rangle$
satisfying~(\ref{d.wx0+x1})
are those which are homogeneous of degree~$1$ in $x;$ i.e.,
lie in the summand $R\,x R$ of~(\ref{d.Rx}).
Indeed, if $w(x)$ had a nonzero component in one of the higher degree
summands in~(\ref{d.Rx}), then on substituting $x_0+x_1$ for $x,$ one
of the nonzero components we would get in the left-hand side
of~(\ref{d.wx0+x1}) would lie in a summand of~(\ref{d.Rx0x1}) that
involved both $x_0$ and $x_1,$ while this is not true of the
right-hand side of~(\ref{d.wx0+x1}).
On the other hand, if $w(x)$ had a nonzero component $a$ in degree zero,
then the degree-zero
component of the left-hand side of~(\ref{d.wx0+x1}) would
be $a,$ while that of the right-hand side would be $2a.$
So $w(x)$ is homogeneous of degree~$1;$ i.e., we may write
\begin{equation}\label{d.wx=sum}
w(x)\ =\ \sum_1^n a_i\,x\,b_i
\end{equation}
for some $a_1,\dots,a_n,\ b_1,\dots,b_n\in R.$
This necessary condition for~(\ref{d.wx0+x1}) to hold is sufficient
as well; in fact, it clearly implies that the functions induced
by $w(x)$ respect the $\!K\!$-module structure.

It remains to bring in the conditions that the operation induced by
$w(x)$ respect $1,$ and respect multiplication.
The former condition says that
\begin{equation}\label{d.1|->1}
w(1)\ =\ 1,
\end{equation}
i.e.,
\begin{equation}\label{d.sum=1}
\sum_1^n a_i\,b_i\ =\ 1,
\end{equation}
while the latter condition,
\begin{equation}\label{d.wx0x1}
w(x_0\,x_1)\ =\ w(x_0)\ w(x_1)\,,
\end{equation}
translates to
\begin{equation}\label{d.sumsum}
\sum_{i=1}^n\ a_i\,x_0\,x_1\,b_i\ =
\ \sum_{j=1}^n\ \sum_{k=1}^n\ a_j\,x_0\,b_j\,a_k\,x_1\,b_k\,.
\end{equation}

To study these conditions, let us now assume that $K$ is a field.
In that case, if there is any $\!K\!$-linear dependence relation
among the coefficients $a_1,\dots,a_n$ in~(\ref{d.wx=sum}),
then we can rewrite one of these elements
as a $\!K\!$-linear combination of the rest, substitute
into~(\ref{d.wx=sum}),
collect terms with the same left-hand factor, and thus
transform~(\ref{d.wx=sum}) into an expression of
the same form, but with a smaller number of summands.
We can do the same if there is a $\!K\!$-linear relation
among $b_1,\dots,b_n.$
Hence, if we choose the expression~(\ref{d.wx=sum}) to
minimize $n,$ we get
\begin{equation}\label{d.lin_indep}
a_1,\dots,a_n\ \ \mbox{are $\!K\!$-linearly independent, and}
\ \ b_1,\dots,b_n\ \ \mbox{are $\!K\!$-linearly independent.}
\end{equation}

Now let $A$ be any $\!K\!$-vector-space
basis of $R$ containing $a_1,\dots,a_n,$
and $B$ any basis containing $b_1,\dots,b_n.$
Then as a $\!K\!$-vector-space, the summand
$R\,x_0\,R\,x_1\,R\cong R\otimes R\otimes R$ of~(\ref{d.Rx0x1}),
in which the two sides of~(\ref{d.sumsum}) lie, decomposes as a
direct sum $\bigoplus_{a\in A,\,b\in B}\,a\,x_0\,R\,x_1\,b.$
If for each $j$ and $k$ we take take the component of~(\ref{d.sumsum})
in $a_j\,x_0\,R\,x_1\,b_k\cong R,$ and drop the outer factors
$a_j\,x_0$ and $x_1\,b_k,$ we get the equation in~$R,$
\begin{equation}\label{d.*dij}
\delta_{jk}\ =\ b_j\,a_k\quad (j,k=1,\dots,n).
\end{equation}

What this says is that if we write $a$ for the row
vector over $R$ formed by $a_1,\dots,a_n,$ and $b$ for the column
vector formed by $b_1,\dots,b_n,$ then $b\,a$ is the identity
matrix $I_n.$
On the other hand,~(\ref{d.sum=1}) says that $a\,b$ is the
$1\times 1$ identity matrix $I_1.$
Thus, regarding these vectors as describing homomorphisms
of right $\!R\!$-modules $a:R^n\to R$ and $b:R\to R^n,$
these relations say that $a$ and $b$ constitute an isomorphism
\begin{equation}\label{d.n=1}
R^n\ \cong\ R\quad\mbox{as right $\!R\!$-modules}.
\end{equation}

For many sorts of rings $R$ (e.g., any ring admitting a homomorphism
into a field),~(\ref{d.n=1}) can only hold for $n=1.$
In such cases, $a$ and $b$ become mutually inverse
elements, so~(\ref{d.wx=sum}) takes
the form $w(x)=a\,x\,a^{-1},$ and our inner
endomorphism is an inner automorphism in the classical sense.
The element $a$ such that $w(x)=a\,x\,a^{-1}$
is easily seen to be determined up to a scalar factor
in $K,$ so the group of extended inner automorphisms of $R$ is
isomorphic to the quotient group of the units of $R$ by the
units of~$K.$

On the other hand, there are rings $R$ admitting
isomorphisms~(\ref{d.n=1})
for $n>1$ \cite{WGL}, \cite{IBN}, \cite{ALSC}, \cite{coproducts2}.
If in such an $R$ we take a row vector
$a$ and column vector $b$ describing such an isomorphism,
then by the above computations, the
element $w(x)=\sum a_i\,x\,b_i$ determines an unfamiliar
sort of extended inner endomorphism of $R.$
It is not hard to verify that this system of maps
can be described as follows.

Since (for any ring $R)$
the ring of endomorphisms of the right $\!R\!$-module
$R^n$ is isomorphic to the \mbox{$n\times n$}
matrix ring $M_n(R),$ a module isomorphism~(\ref{d.n=1})
yields a $\!K\!$-algebra isomorphism $M_n(R)\cong M_1(R).$
Moreover, for every object $f:R\to S$ of $(R\downarrow\Rg_K),$ the
vectors $a$ and $b$ over $R$ induce vectors $f(a),\ f(b)$ over $S$
satisfying the same relations, and hence likewise inducing
isomorphisms of matrix rings.
The endomorphism
of $S$ induced by $w(x)$ can now be described as the composite
\begin{equation}\label{d.S>M_n>S}
S\ \xrightarrow{\r{diag.}}\ M_n(S)
\ \xrightarrow[\cong]{((r_{ij}))\ \mapsto \ \sum f(a_i)\,r_{ij}\,f(b_j)}
\ S\,.
\end{equation}
Since the right-hand arrow in~(\ref{d.S>M_n>S}) is bijective, the
composite arrow will, like the left-hand arrow, always be one-to-one,
but will not be surjective for any nonzero $S$ unless $n=1;$
so the latter is the only case where the above
construction gives {\em automorphisms} of the algebras $S.$

These observations are summarized below, along with a final assertion
which the reader should not find hard to verify, which corresponds
to a description of the degree of
nonuniqueness of the expression for an element $w=\sum a_i\otimes b_i$
in a tensor product of $\!K\!$-vector-spaces, when written using
the smallest number of summands (the rank of the element as a tensor);
equivalently, using $\!K\!$-linearly independent $a_i$ and $b_j.$
Note that~(\ref{d.*dij}), which we deduced using those conditions
of $\!K\!$-linear independence, clearly also implies them.

\begin{theorem}\label{T.end_R-Rg}
Let $K$ be a field, and $R$ a nonzero $\!K\!$-algebra.
Then for every extended inner {\em automorphism} $(\beta_f)$ of $R,$
there is an invertible element $a\in R,$ unique up to a
scalar factor, such that for each $f:R\to S,$ the automorphism
$\beta_f$ of $S$ is given by conjugation by $f(a).$

More generally, each extended inner {\em endomorphism} of $R$ has
the form\textup{~(\ref{d.S>M_n>S})} for a pair $(a,b),$ where
for some $n,$ $a=(a_i)$ is a length-$\!n\!$ row vector over $R,$
and $b=(b_i)$ a height-$\!n\!$ column vector,
satisfying\textup{~(\ref{d.sum=1})} and\textup{~(\ref{d.*dij})},
equivalently, describing an isomorphism~\textup{(\ref{d.n=1})}.
Two such pairs of vectors $(a,b)$ and $(a',b'),$ associated with
integers $n$ and $n'$ respectively, determine the
same extended inner endomorphism if and only if $n=n'$ and
there exists some $U\in\r{GL}(n,K)$ such that
\begin{equation}\label{d.aUb}
a'\ =\ a\,U,\quad b'\ =\ U^{-1}b.
\end{equation}
\qed
\end{theorem}

The conclusion $n=n'$ in the above result follows from the
uniqueness of $w(x),$ and hence of its rank as a member of
$R\otimes R;$ but let us note a way to see it directly, and
in fact to see that $n$ is determined by the value of our
extended inner endomorphism at any nonzero object $f:R\to S$ of
$(R\downarrow\Rg_K).$
From~(\ref{d.S>M_n>S}) we see that the {\em centralizer} in $S$ of the
image of our extended inner endomorphism will be
isomorphic to $M_n(Z(S))$
as a $\!Z(S)\!$-algebra, where $Z(S)$ is the center of $S.$
In particular, it will be free of rank $n^2$ as a module
over $Z(S);$ and free modules over commutative rings have unique rank.

We have noted that~(\ref{d.S>M_n>S}) shows that every extended inner
endomorphism $(\beta_f)$ of
$R$ consists of {\em one-to-one} endomorphisms $\beta_f.$
This too can be seen from elementary considerations:
Any $\!K\!$-algebra $S$ can be embedded in
a simple $\!K\!$-algebra $T;$
and any endomorphism of $S$ arising from
an extended inner endomorphism of $R$ will then extend to an
endomorphism of $T,$ which necessarily has trivial kernel.
(The embeddability of any $\!K\!$-algebra in a simple $\!K\!$-algebra
was proved in \cite[Corollary~1 and Remark~2]{bokut}.
A different method of getting such an embedding, noted for
Lie algebras in \cite[Theorem~B]{stewart},
is also applicable to associative algebras.)

It is not hard to add to Theorem~\ref{T.end_R-Rg}
the necessary and sufficient condition for two extended
inner endomorphisms of $R$ as in the final statement to agree,
not necessarily globally, but at $R,$ i.e., to
determine the same inner endomorphism of $R.$
The condition has the same form as~(\ref{d.aUb}),
but with $U$ now taken in $\r{GL}(n,Z(R)).$
For $n=1,$ this is the expected condition that the conjugating elements
differ by an invertible central factor in $R.$\vspace{.5em}

(For the reader familiar with~\cite[Chapter~III]{coalg} we remark that
$(R\downarrow\Rg_K)$ is the category there called $R\,\mbox{-}\Rg_K,$
and that the $R\langle x\rangle$ occurring in the above arguments
is the underlying algebra of the coalgebra object representing
the {\em forgetful} functor $R\,\mbox{-}\Rg_K\to\Rg_K.$
Since the values of that forgetful functor have, in particular,
additive group structures, the functor can
be regarded as $\!\fb{Ab}\!$-valued, so by
\cite[Theorem~13.15 and Corollary~14.8]{coalg}, its representing
$\!K\!$-algebra is freely generated over $R$ by an $\!(R,R)\!$-bimodule.
This is the $R\,x\,R\cong R\otimes R$ of~(\ref{d.Rx}).
Our extended inner endomorphisms of $R$ correspond to
endomorphisms of $R\langle x\rangle$ as a co-ring.
Since these are in particular co-abelian-group endomorphisms, they will
be induced by bimodule endomorphisms of $R\,x\,R;$
this is the content of~(\ref{d.wx=sum}).
Our subsequent arguments determine when
such an endomorphism respects the counit and comultiplication
of $R\langle x\rangle.)$

\section{What if $K$ is not a field?}\label{S.non-field}

For a general commutative ring $K$ and an arbitrary
object $R$ of $\Rg_K,$ any vectors $a,\ b$ over $R$ that
satisfy~(\ref{d.sum=1}) and~(\ref{d.*dij}) will still
yield an element $w(x)=\sum_1^n a_i\,x\,b_i$
inducing an extended inner endomorphism~(\ref{d.S>M_n>S})
of $R$ in $\Rg_K;$ but we can no longer say that every
extended inner endomorphism has this form.
As an easy counterexample, if $K$ is a direct product $K_1\times K_2$
of two fields, then $\Rg_K\cong\Rg_{K_1}\times\Rg_{K_2},$ and
one can show that any extended inner endomorphism of
an object $R_1\times R_2$ of $\Rg_K$ $(R_i\in\Rg_{K_i})$ is determined
by an extended inner endomorphism of $R_1$
and an extended inner endomorphism of $R_2.$
Now if $R_1$ and $R_2$ are both nonzero, and if they
respectively admit extended inner endomorphisms
$(\beta_{1,f})$ and $(\beta_{2,f}),$ associated with distinct
positive integers $n_1$ and $n_2,$ then these together
induce an extended inner endomorphism of $R$ which does not
have the form~(\ref{d.S>M_n>S}) for any $n.$

For a different sort of example, suppose $K$ is a commutative integral
domain having a nonprincipal invertible ideal $J,$ and let
$F$ be the field of fractions of $K.$
(Recall that an ideal $J$ of $K$ is called invertible if it
has an inverse in the multiplicative monoid of {\em fractional
ideals} of $K,$ that is, nonzero $\!K\!$-submodules of $F$
whose elements admit a common denominator.
The integral domains $K$ all of whose nonzero ideals
are invertible are the Dedekind domains \cite[Theorem~9.8]{A+M}.
Thus, any Dedekind domain that is not a PID has a nonprincipal
invertible ideal $J.)$
Suppose we
form the Laurent polynomial ring in one indeterminate, $F[\,t,t^{-1}],$
and within this, let $R$ be the subring $K[\,Jt,\,J^{-1}t^{-1}].$
Then in $R\langle x\rangle,$ the $\!K\!$-submodule
$J\,t\,x\,J^{-1}t^{-1}\cong J\otimes J^{-1}\cong K$ is free on one
generator, which we shall
call $w(x),$ and which we might write (in)formally as $t\,x\,t^{-1},$
though $t$ itself is not an element of $R.$
One finds that $w(x)$ satisfies (\ref{d.wx0+x1}), (\ref{d.1|->1})
and~(\ref{d.wx0x1}), and so induces an extended inner endomorphism;
but ``$t\,x\,t^{-1}$'' does not have the form $s\,x\,s^{-1}$ for any
invertible element $s\in R,$ so this extended inner endomorphism is
not as described in Theorem~\ref{T.end_R-Rg}.
Incidentally, this extended inner endomorphism has an inverse,
induced by ``$t^{-1}x\,t$'', so it is even an extended
inner automorphism (showing that the first half of our
title is not {\em quite} true).

In taking an example of maximal simplicity, we have ended up
with a commutative $R,$ so that the automorphism of $R$ itself
induced by the above extended inner automorphism is trivial, and
can be described as conjugation by $1\in R.$
To avoid this, let us freely adjoin to
the $\!F\!$-algebra $F[\,t,t^{-1}]$ another
noncommuting indeterminate, $u,$ getting the algebra
$F\langle t,\,t^{-1},\,u\rangle,$ and within this take
$R=K\langle Jt,\,J^{-1}t^{-1},\,u\rangle.$
Then the automorphism of $R$ induced by ``$t\,x\,t^{-1}$'' is now
nontrivial, and is still not inner in the classical sense;
in particular, it
takes $u$ to $t\,u\,t^{-1},$ though conjugation by no invertible
element of $R$ can do this.

Our general result for $K$ a field, and the above examples for
other sorts of $K,$ can be subsumed in a common construction:
Suppose $P$ is a $\!K\!$-module,
and $R$ a $\!K\!$-algebra having an isomorphism
\begin{equation}\label{d.PR>R}
a: P\otimes R\ \stackrel{\cong}{\longrightarrow}\ R
\end{equation}
as right $\!R\!$-modules.
(In the case where $K$ was a field, $P$ was an $\!n\!$-dimensional
vector space; in our $K_1\times K_2$ example, it was the
module $K_1^{n_1}\times K_2^{n_2};$
in the $K[Jt,\,J^{-1}t^{-1}]$ and
$K\langle Jt,\,J^{-1}t^{-1},\,u\rangle$ examples, it is $J.$
In this last case, one has an isomorphism~(\ref{d.PR>R})
$J\otimes R\cong R$ because $J\otimes R\cong JR=t^{-1}R\cong R,$
the middle equality holding because $R$ is closed in
$F[\,t,t^{-1}]$ under multiplication by $J^{-1}t^{-1}$ and $J\,t.)$
Such a map~(\ref{d.PR>R}) yields,
for every algebra $S$ with a homomorphism
$R\to S,$ a $\!K\!$-algebra homomorphism
\begin{equation}\label{d.S>End>S}
S\ \cong\ \r{End}_S(S_S)
\ \xrightarrow{P\otimes-}
\ \r{End}_S(P\otimes S_S)
\ \xrightarrow[\cong]{a\otimes_R\,-}\ \r{End}_S(S_S)\ \cong\ S.
\end{equation}
The $\!K\!$-module $P$ in this construction need not be unique.
For instance, if we take an example based on an
isomorphism~(\ref{d.n=1}), but where our $R$ is an algebra over
some epimorph $K'$ of $K$ (in the category-theoretic sense; e.g.,
a factor-ring or a localization), then regarding $R$ as a
$\!K\!$-algebra, we could choose the $\!K\!$-module $P$
of~(\ref{d.PR>R}) to be either $K^n$ or $K'^n.$

In all the cases looked at so far, our $\!K\!$-module $P$ either was,
or (in the above paragraph) could be taken to be, projective over $K.$
But there are examples where this is impossible:  Consider
any integral domain $K$ which has an
epimorph of the form $K_1\times K_2$ for fields $K_1$ and $K_2$
(e.g., $\mathbb{Z}$ has such homomorphic images).
Then if we construct, as in the first paragraph of
this section, an algebra $R$ over $K_1\times K_2$
and an extended inner endomorphism of $R$
based on a $\!K'\!$-module $P = K_1^{n_1}\times K_2^{n_2}$
with $n_1\neq n_2,$ this cannot arise from an example based on
a projective $\!K\!$-module.
This follows from the fact that
for a finitely generated projective module over an integral domain
$K,$ the rank is constant as a function on the prime spectrum of $K$
\cite[Ch.2, \S5, n\raisebox{.3em}{\scriptsize{o}}.\,2, Theor\`{e}me 1,
(a)$\Rightarrow$(c)]{AlgComm},
\cite[p.53, Exercise~22]{TYL}.

\begin{question}\label{Q.non-field}
If $K$ is a commutative ring and $R$ a nonzero object
of $\Rg_K,$ can every extended inner endomorphism of $R$ be obtained
as in~\textup{~(\ref{d.S>End>S})} from a module
isomorphism~\textup{~(\ref{d.PR>R})}?
\end{question}

The nonuniqueness of the $P$ in the above construction makes me dubious.

We saw in the preceding section that for $K$ a field, all
inner endomorphisms of $\!K\!$-algebras were one-to-one.
In an appendix, \S\ref{S.1-1}, we show
that the same is true for any $K.$

\section{Other concepts of ``inner endomorphism'' in the literature.}\label{S.literature}

A MathSciNet search for ``inner endomorphism'' leads
to a number of concepts, some
of which have interesting overlaps with the one we have been studying.

A striking case, to which we alluded in the
abstract, comes from the theory of $\!C^*\!$-algebras.
If $H$ is a Hilbert space, and $\mathcal{B}(H)$
the $\!C^*\!$-algebra of bounded operators $H\to H,$
it is shown in \cite[Proposition~2.1]{arveson} that {\em every}
endomorphism of the $\!C^*\!$-algebra $\mathcal{B}(H)$ has a
form analogous to what we found in Theorem~\ref{T.end_R-Rg}, namely
\begin{equation}\label{d.VAV*}
A\ \mapsto\ \sum V_i\,A\,V_i^*,
\end{equation}
where the $V_i$ are a (possibly infinite) family of
isometric embeddings $H\to H$ having mutually orthogonal ranges
which sum to $H,$ and $V_i^*$ is the adjoint of $V_i.$

Here is a heuristic sketch for the algebraist of why this is plausible.
Since complex Hilbert spaces look alike except for their dimension,
it is natural to generalize the problem of characterizing
endomorphisms of $\mathcal{B}(H)$ to that of characterizing
homomorphisms $\mathcal{B}(H_1)\to \mathcal{B}(H_2)$ for two
Hilbert spaces $H_1$ and $H_2.$
If $H_1$ and $H_2$ are finite-dimensional,
of dimensions $d_1$ and $d_2,$ then
$\mathcal{B}(H_1)$ and $\mathcal{B}(H_2)$ are matrix algebras
$M_{d_1}(\mathbb{C})$ and $M_{d_2}(\mathbb{C}).$
Temporarily ignoring the $\!C^*\!$ structure, we know that a
$\!\mathbb{C}\!$-algebra homomorphism
$M_{d_1}(\mathbb{C})\to M_{d_2}(\mathbb{C})$
exists if and only if $d_2= n\,d_1$ for some integer $n,$ and
that in this case, it can be gotten by writing $\mathbb{C}^{d_2}$ as
a direct sum of $n$ copies of $\mathbb{C}^{d_1},$ and letting
$M_{d_1}(\mathbb{C})$ act in the natural way on each of these.
If $a_1,\dots,a_n: \mathbb{C}^{d_1}\to\mathbb{C}^{d_2}$
are the chosen embeddings
and $b_1,\dots,b_n: \mathbb{C}^{d_2}\to\mathbb{C}^{d_1}$
the corresponding projections, the induced
map $M_{d_1}(\mathbb{C})\to M_{d_2}(\mathbb{C})$ is given by
\begin{equation}\label{d.airbi}
r\ \mapsto\ \sum a_i\,r\,b_i.
\end{equation}
If one wants this to be a homomorphism of $\!C^*\!$-algebras, one
has the additional requirement that the $a_i$ each map $H_1$ into $H_2$
isometrically, with orthogonal images; the projections $b_i$
will then be the adjoints of the $a_i.$
Now if instead of finite-dimensional Hilbert spaces we take
an infinite-dimensional Hilbert space $H,$ and
let $H_1=H_2=H,$ then for both finite and infinite $n,$ there exist
expressions of $H$ as a direct sum
(in the infinite case, a completed direct sum) of $n$ copies of itself.
The result of~\cite{arveson} says that all endomorphisms
of $\mathcal{B}(H)$ are expressible essentially as in the finite
dimensional case, in terms of such direct sum decompositions of $H.$

For $R$ any $\!C^*\!$-algebra, not necessarily of
the form $\mathcal{B}(H),$ a family of elements
$V_1,\dots,V_n\in R$ $(n<\infty)$ satisfying the $\!C^*\!$-algebra
relations corresponding to the conditions stated
following~(\ref{d.VAV*}) is equivalent to a homomorphism into $R$ of the
$\!C^*\!$-algebra presented by those generators
and relations; this $\!C^*\!$-algebra is denoted $\mathcal{O}_n.$
The objects $\mathcal{O}_n$ are called Cuntz
algebras, having been introduced by J.\,Cuntz~\cite{Cuntz}.
Since the above construction with $n=1$
gives inner automorphisms of $R$ in the classical sense,
endomorphisms of the form~(\ref{d.VAV*}) in a general $\!C^*\!$-algebra
(where they are not in general the only endomorphisms)
are called inner {\em endo}\/morphisms.

(For $n=\infty,$ things are not as neat.
Though in $\mathcal{B}(H),$ the infinite sums~(\ref{d.VAV*})
converge in a topology obtained from the Hilbert space $H,$ this
is not the topology arising from the $\!C^*\!$-norm on $\mathcal{B}(H).$
In defining the $\!C^*\!$-algebra $\mathcal{O}_\infty$
one has to omit the relation $\sum V_i\,V_i^*=1,$
because the infinite sum will not converge;
and maps of this object into a
$\!C^*\!$-algebra $R$ do not induce endomorphisms of $R,$
though they are still of interest.)
\vspace{.5em}

The next concept I will describe is not called an ``inner endomorphism''
by the author who studies it, though it did turn up in a MathSciNet
search for that phrase.
(In the paper in question,
``inner endomorphism'' is used for ``endomorphism of a subalgebra''.)
Namely, in~\cite{terms}, if $A$ is an algebra in the sense of
universal algebra, a {\em termal endomorphism} of $A$ means an
endomorphism $\alpha$ which is
expressible by a {\em term} in one variable $x,$ i.e., in
the notation of this note,
a word $w(x)$ in the operations of $A$ and constants taken from~$A.$

Note that if such a word $w(x)$ defines an endomorphism of $A,$
i.e., if the set map it determines respects all operations of $A,$
then this fact is equivalent to a family of identities in
the operations of $A$ and the constants occurring in $w.$
If $\V$ is some variety containing $A,$ those
identities need not be satisfied by all members of
$(A\downarrow\V),$ so $w$ may not define what we
are calling a $\!\V\!$-extended inner endomorphism of $A.$
However, if we regard $(A\downarrow\V)$ as a variety,
with the images of the elements of $A$ as new zeroary operations,
then the identities named will define a subvariety
$\V_0\subseteq(A\downarrow\V),$ on which $w(t)$ does
induce an extended inner endomorphism of $\r{id}_A,$
and hence an inner endomorphism of $A.$

The phrase ``inner endomorphism'' has in fact been used
in the theory of semigroups \cite{semigp}, \cite{invsemigp}
to describe some particular classes of what~\cite{terms} calls
termal endomorphisms.\vspace{.5em}

A different use of the phrase ``inner endomorphism'' has occasionally
been made in group theory.
Observe that if $G$ is a group, and $\alpha: G\to G$ is a set map which
in one or another sense can be ``approximated'' arbitrarily
closely by endomorphisms, then in general, $\alpha$ will
again be an endomorphism; but that if the approximating endomorphisms
are bijective, this does not force $\alpha$ to be bijective.
In such situations, if the approximating maps are inner
automorphisms, $\alpha$ has been called an ``inner
endomorphism'', preceded by some qualifying adverb.
Specifically, if one can find inner automorphisms of $G$ that
agree with $\alpha$ on a directed family of subgroups
having $G$ as union (though the conjugating elements need not belong
to the corresponding subgroups, so that $\alpha$ need not carry those
subgroups onto themselves), then $\alpha$ is called (in \cite{locally},
and \cite[p.\,201, starting in paragraph before Theorem~5.5.9]{sylow})
a ``locally inner endomorphism'',
while if $\alpha$ induces inner automorphisms on a
class of homomorphic images of $G$ that separates points, it
is called in~\cite{residually} a ``residually inner endomorphism''.
In the same spirit,
\cite{asymp} calls a topological limit of inner automorphisms of
a $\!C^*\!$-algebra an ``asymptotically inner endomorphism''
(a usage apparently unrelated to the sense of ``inner endomorphism''
of a $\!C^*\!$-algebra described above).

On a somewhat related theme,~\cite{AH} takes a finite-dimensional
associative unital algebra $A$ over a field $K,$
with $\!K\!$-vector-space basis $\{u_1,\dots,u_n\},$ forms an
extension field $K_0$ of $K$ by adjoining $n$ algebraically
independent elements, uses these as coefficients in forming
a ``generic'' element of the $\!K_0\!$-algebra $A\otimes K_0,$
and notes that this element will necessarily be invertible,
so that conjugation by it may be thought of
as a ``generic'' inner automorphism of $A.$
It is then noted that for certain elements $a\in A,$
the specialization of our indeterminates to the coefficients of the
$u_i$ in $a$ may turn the above conjugation map into a map that is
everywhere defined on $A,$ even if $a$ itself was not invertible.
(Intuitively, the map obtained
by that specialization is approximated by the operations of
conjugation by nearby invertible elements.)
The resulting maps are endomorphisms, but examples are
given showing that they may not be automorphisms, and they are named
``inner endomorphisms'' of $A.$

I don't see a direct relation between the concepts cited in the
last two paragraphs and those of this paper.
However, pondering the idea of~\cite{AH}, in which one
performs a conjugation $r\mapsto a\,r\,a^{-1}$ for which,
from the point of view of $A,$ the pair $(a,a^{-1})$
``doesn't quite exist'', helped lead me to
the example of the preceding section, in which a conjugating element
$t$ was put out of reach by multiplying
by a nonprincipal invertible ideal $J\subseteq K.$

I will note another use of ``inner'' in the literature, not restricted
to endomorphisms, at the end of \S\ref{S.deriv_def}.
\vspace{.5em}

We now return to inner endomorphisms in the sense of
Definition~\ref{D.inner}.

\section{Extended inner endomorphisms in other categories of algebras -- some easy observations.}\label{S.other}
We have examined extended inner endomorphisms in $\Gp$ and $\Rg_K.$
What about other categories of algebras?

In the category $\fb{Ab}$ of abelian groups (which we will write
additively), the result of adjoining a ``generic'' element $x$ to an
object $A$ is $A\oplus\langle x\rangle,$ each element $w(x)$ of which
has the form $a+nx$ for unique $a\in A$ and $n\in\mathbb{Z}.$
Clearly, the system of operations induced by this element
will respect the group operations of arbitrary objects
of $(A\downarrow\fb{Ab})$ if and only if $a=0;$
so here the general extended inner endomorphism is
given by multiplication by a fixed integer $n;$ it will be
an extended inner automorphism if and only if $n=\pm 1.$
These are very different from the extended inner
endomorphisms of the same group $A$ in the larger category $\Gp.$

Note that the above extended inner endomorphisms of $A$ do not
really depend on $A.$
Though we are looking at them
as endomorphisms of the forgetful functor
$(A\downarrow\fb{Ab})\to\fb{Ab},$ they are induced by
endomorphisms of the identity functor of $\fb{Ab}.$
We might call such operations {\em absolute} endomorphisms.

We can answer in the negative the analog
of Question~\ref{Q.group} with $\fb{Ab}$ in place of $\fb{Group}.$
Let $p$ be a prime, and let $A=Z_{p^\infty},$ the $\!p\!$-torsion
subgroup of $\mathbb{Q}/\mathbb{Z}.$
Recall that this abelian group is injective, that its nonzero
homomorphic images are all isomorphic to it, and that its endomorphism
ring is canonically isomorphic to the ring of $\!p\!$-adic integers.
It is easy to see that the action of each $\!p\!$-adic
integer $c$ on $A$ makes a commuting square with the
action of $c$ on every homomorphic image $f(A).$
Now if $f$ is a homomorphism of $A$ into any abelian group $B,$
the injectivity of $f(A)$ implies that $B$ can
be decomposed as $f(A)\oplus B_0;$ hence the action of $c$ on
$f(A)$ can be extended to an action on $B;$ e.g.,
by using the identity on $B_0.$
It follows that all the endomorphisms of $A$
(including its uncountably many automorphisms) have
the one-$\!B\!$-at-a-time extendibility property analogous to the
hypothesis of Question~\ref{Q.group}, though we have seen that only
those corresponding to multiplication by integers are inner,
as defined in Definition~\ref{D.inner}.
Hence in $\fb{Ab},$ the one-$\!B\!$-at-a-time
extendibility property is
strictly weaker than the functorial extendibility property
by which we have defined inner endomorphisms and automorphisms.

It would be interesting to investigate inner automorphisms
and endomorphisms in still other varieties of groups.\vspace{.5em}

In the category of commutative rings, it is not hard to verify
that $\mathbb{Z}$ has no nontrivial extended inner endomorphisms.
On the other hand, $\mathbb{Z}/p\<\mathbb{Z}$ has,
for every positive integer $n,$ the extended
inner endomorphism given by exponentiation by $p^n$
(the $\!n\!$-th power of the Frobenius map).
These endomorphisms are trivial on $\mathbb{Z}/p\<\mathbb{Z}$ itself;
but on every other integral domain of characteristic $p,$
the Frobenius map is a nontrivial inner endomorphism.\vspace{.5em}

If $A$ is an object of the variety of abelian
semigroups (written multiplicatively),
and $e$ an idempotent element of $A,$ then multiplication
by $e$ is an inner endomorphism;
the same is true in the category of nonunital commutative rings.
Similarly, if $D$ is an object of the category of distributive
lattices, then for any $a,b\in D,$
the operators $a\vee -,$ $b\wedge -,$ and
$a\vee(b\wedge -)$ are inner endomorphisms.

If $A$ is an object of the category of all semigroups (not
necessarily abelian), and
$e$ is a {\em central} idempotent of $A,$ then the word $w(x)=ex$
gives a {\em termal} endomorphism of $A$ in the sense of~\cite{terms}
(see preceding section), but not an inner endomorphism in our sense.
However, following the the idea noted in that section, if we
form the subvariety of $(A\downarrow\fb{Semigroup})$ defined by the
identity making the image of $e$ central, then $w(x)=ex$ does
determine an inner extended endomorphism in that category.
The analogous observations hold for nonunital commutative rings, and for
not necessarily distributive lattices.

\section{Derivations of associative algebras.}\label{S.d_assoc}

Alongside inner automorphisms of groups and rings, there is
another pair of cases where the modifier ``inner'' is classical:
inner {\em derivations} of associative and Lie algebras.
We shall examine the case of associative algebras in this section,
that of Lie algebras in \S\ref{S.Lie}.

If $K$ is a commutative ring and $R$ an object of $\Rg_K,$ we recall
that a {\em derivation} of $R$ as a $\!K\!$-algebra
means a set-map $d: R\to R$ satisfying
\begin{equation}\label{d.dr+s}
d(r+s)\ =\ dr\ +\ ds\quad (r,s\in R),
\end{equation}
\begin{equation}\label{d.dcr}
d(c\,r)\ =\ c\,dr\quad (c\in K,\ r\in R),
\end{equation}
\begin{equation}\label{d.drs}
d(r\,s)\ =\ d(r)\,s\ +\ r\,d(s)\quad (r,s\in R).
\end{equation}

In particular, for every $t\in R,$ the map $d$ defined by
\begin{equation}\label{d.inner_der}
d\,r\ =\ t\,r-r\,t
\end{equation}
is a derivation of $R,$
called the {\em inner derivation} induced by $t,$
and written $tr-rt=[t,\,r].$

Such an inner derivation $d$ clearly has the analog of the property
of inner automorphisms of groups which we abstracted in
Definition~\ref{D.inner}; namely, that to every $f:R\to S$ in
$(R\downarrow\Rg_K)$ we can associate a derivation
$d_f$ of $S,$ in such a way that
\begin{equation}\label{d.d_id}
d_{\r{id}_R}\ =\ d,
\end{equation}
and that given two objects $f_i:R\to S_i$ $(i=1,2)$
of $(R\downarrow\Rg_K)$ and a morphism $h:S_1\to S_2$ in that
category, we have
\begin{equation}\label{d.dfunct}
d_{f_2}\,h\ =\ h\,d_{f_1}.
\end{equation}

What about the converse?
Given a system of derivations $d_f$ satisfying~(\ref{d.dfunct}), let us,
as in our investigation of automorphisms and endomorphisms,
look at their action on a generic element.
Let $\eta:R\to R\langle x\rangle$ be the natural inclusion and
write $d_\eta(x)=w(x)\in R\langle x\rangle.$
As before,~(\ref{d.dr+s}) implies that
\begin{equation}\label{d.dx=sum}
w(x)\ =\ \sum_1^n a_i\,x\,b_i
\end{equation}
for some $a_1,\dots,a_n,\ b_1,\dots,b_n\in R,$ and conversely,
this condition implies both~(\ref{d.dr+s}) and~(\ref{d.dcr}).
To handle~(\ref{d.drs}), we need, as before, an additional
assumption; but this time we can get away with much less than
$K$ being a field.
Let us merely assume that the canonical map $K\to R$ makes $K$ a
$\!K\!$-module direct summand in $R;$ i.e., that there exists a
$\!K\!$-module-theoretic left inverse $\varphi:R\to K$ to that map.
Given such a $\varphi,$ it is not hard to see that we can obtain
from~(\ref{d.dx=sum}) an equation of the same form
(possibly with $n$ increased by~$1)$ in which $a_1=1,$
while $a_2,\dots,a_n\in\r{Ker(\varphi)}.$
So let us assume that~(\ref{d.dx=sum}) has those properties.

Let us now take the generic instance of~(\ref{d.drs}),
namely, in $R\langle x_0,\,x_1\rangle,$ the equation
\begin{equation}\label{d.dx0x1}
\sum_1^n a_i\,x_0\,x_1\,b_i\ =
\ (\sum_1^n a_i\,x_0\,b_i)\ x_1\ +\ x_0\ (\sum_1^n a_i\,x_1\,b_i)\,.
\end{equation}
The two sides of this equation lie in
$R\,x_0\,R\,x_1\,R\cong R\otimes R\otimes R.$
Let us apply $\varphi$ to the leftmost of the three tensor factors,
getting an equation in $x_0\,R\,x_1\,R,$
and take the right coefficient of $x_0$ therein.
This is an equation in $R\otimes R\cong R\,x_1\,R,$ namely
\begin{equation}\label{d.pre-d_inner}
x_1\,b_1\ =\ b_1\,x_1\ +\ \sum_1^n a_i\,x_1\,b_i\,.
\end{equation}
Solving for the summation, which is $w(x_1),$
and writing $x$ in place of $x_1,$ we get
\begin{equation}\label{d.d_inner}
w(x)\ =\ x\,b_1\ -\ b_1\,x\,.
\end{equation}

Hence, each map $d_f$ is the inner derivation,
in the classical sense, determined by $f(b_1).$
We summarize this result below.
The ``unique up to \dots'' assertion in the conclusion is obtained by
noting that an element $b\in R$ satisfies $b\otimes 1-1\otimes b=0$ in
$R\otimes R=(K\oplus\r{Ker}(\varphi))\otimes(K\oplus\r{Ker}(\varphi))$
if and only if the component of $b$
in $\r{Ker}(\varphi)$ is $0;$ i.e., if and only if $b\in K.$

\begin{theorem}\label{T.d_inner}
Let $K$ be a commutative ring and $R$ a $\!K\!$-algebra, and
suppose we have a function associating to every $f:R\to S$
in $(R\downarrow\Rg_K)$ a derivation $d_f$ of the $\!K\!$-algebra $S,$
such that\textup{~(\ref{d.dfunct})} holds for every morphism
$h$ of $(R\downarrow\Rg_K).$
Suppose also that the canonical map $K\to R$ has
a $\!K\!$-module-theoretic left inverse.

Then there exists $b\in R,$
unique up to an additive constant from $K,$ such that for each
$f:R\to S,$ $d_f$ is the inner derivation of $S$ induced by $f(b).$\qed
\end{theorem}

We can push this a bit further.
Instead of assuming that the canonical map $K\to R$ has
a $\!K\!$-module left inverse, assume the $\!K\!$-algebra
structure on $R$ extends to a $\!K'\!$-algebra structure for
some epimorph $K'$ of $K$ in the category of commutative rings,
and that the map of $K'$ into $R$ has a $\!K'\!$-module left inverse.
(This is the same as a $\!K\!$-module left inverse to the latter map.
On the other hand, when the epimorphism $K\to K'$ is not an isomorphism,
the map of $K$ itself into $R$ cannot
have a $\!K\!$-module left inverse.)
Then we can apply the above theorem to $R$ as a $\!K'\!$-algebra;
moreover, it is not hard to show that
$(R\downarrow\Rg_K)\cong(R\downarrow\Rg_{K'});$ so the
characterization by Theorem~\ref{T.d_inner}
of such systems of derivations parametrized
by $(R\downarrow\Rg_{K'})$ gives the same result for systems
of derivations parametrized by $(R\downarrow\Rg_K).$
Note, however, that the element inducing the system
will be unique up to an additive constant in $K',$ rather than in $K.$

I know of no example showing the need for any version
of the module-theoretic hypothesis of Theorem~\ref{T.d_inner}
for the existence half of the conclusion.
So we ask

\begin{question}\label{Q.d_inner}
If $K$ is a commutative ring and $R$ an associative unital
$\!K\!$-algebra, must every system of derivations $(d_f)$
satisfying\textup{~(\ref{d.dfunct})}
be induced, as above, by an element $b\in R$?
\end{question}

\section{How should one define ``extended inner derivation''?}\label{S.deriv_def}
We would have stated Theorem~\ref{T.d_inner}
and Question~\ref{Q.d_inner} in terms of
``extended inner derivations of $R$'', if it were clear
how to define this concept.
We could, of course, make an ad hoc definition of the phrase,
as a system of derivations $d_f$ satisfying the hypothesis
of those statements; but it would be better if we could make it an
instance of a general use of ``extended inner ---\,''.
Derivations are not, in an obvious way, morphisms in a category,
so we cannot use Definition~\ref{D.inner}.
Below, we will note several ways that derivations can be put in a
more general context, then take the one that seems
best as the basis of our definition.\vspace{.5em}

Recall first that there is a well-known characterization of
derivations in terms of algebra homomorphisms.
If $R$ is a $\!K\!$-algebra, let $I(R)$ denote the $\!K\!$-algebra
obtained by adjoining to $R$ a central, square-zero element
$\varepsilon$ (intuitively, an infinitesimal.
Formally, $I$ can be described as the functor of tensoring over
$K$ with $K[\,\varepsilon\mid\nolinebreak\varepsilon^2=0\,].)$
Then it is easy to verify that a set-map $d:R\to R$ is a derivation
if and only if the map $R\to I(R)$ given by
\begin{equation}\label{d.r+*edr}
r\ \mapsto\ r+\varepsilon\,d(r)
\end{equation}
is a $\!K\!$-algebra homomorphism.
Under this correspondence, the {\em inner} derivations,
in the classical sense, correspond
to the homomorphisms obtained by composing the
inclusion $R\to I(R)$ with conjugation by a unit
of the form $1+\varepsilon\,b$ $(b\in R).$

Using this characterization of derivations, we could put our condition
on families of derivations $d_f$ into category-theoretic language.
But I don't see the resulting framework as fitting a
natural wider class of constructions.\vspace{.5em}

A second approach begins by asking, ``Since
the common value of the two sides of~(\ref{d.dfunct})
is neither a derivation of $S_1,$ nor a derivation of $S_2,$
nor a ring homomorphism, what is it?''
It is, in fact, what is known as ``a derivation from $S_1$ to $S_2$
with respect to the homomorphism $h:S_1\to S_2$'';
i.e., a set-map $d:S_1\to S_2$ which satisfies (\ref{d.dr+s}),
(\ref{d.dcr}), and the generalization of~(\ref{d.drs}),
\begin{equation}\label{d.drs_g}
d(r\,s)\ =\ d(r)\,h(s)\ +\ h(r)\,d(s).
\end{equation}

If, now, for every pair of $\!K\!$-algebras $S_1,\ S_2,$
we let $D(S_1,\,S_2)$ denote the set of all pairs $(h,d),$
where $h:S_1\to S_2$ is a $\!K\!$-algebra homomorphism and
$d:S_1\to S_2$ a derivation with respect to $h$ in the
above sense, then $D(-,-)$
becomes a bifunctor $(\Rg_K)^{\r{op}}\times\Rg_K\to\fb{Set},$
having a forgetful morphism $(h,d)\mapsto h$ to the bivariant
hom functor $\r{Hom}: (\Rg_K)^\r{op}\times\Rg_K\to\fb{Set}.$
The set of derivations of a single $\!K\!$-algebra $S$ is
the inverse image of the identity endomorphism
of $S$ under this forgetful map.

Again, however, I don't know of a natural class of constructions
wider than the derivations to which these observations generalize.
Moreover, the concept of an $\!h\!$-derivation $d:S_1\to S_2$
for $h$ a ring homomorphism is in turn a special case of that
of an $\!(h,h')\!$-derivation, for $h,h':S_1\to S_2$
two homomorphisms; such a derivation is a
map satisfying (\ref{d.dr+s}), (\ref{d.dcr}), and
\begin{equation}\label{d.drs_gh}
d(r\,s)\ =\ d(r)\,h'(s)\ +\ h(r)d(s).
\end{equation}
In this context, we again have the concept
of the inner derivation induced
by an element $b\in S_2,$ namely the operation
\begin{equation}\label{d.inner_der_gh}
d\,r\ =\ b\,h'(r)\ -\ h(r)\,b.
\end{equation}
How this generalization might relate to our concept of ``extended inner
derivation'' is not clear.\vspace{.5em}

A third approach is to start with any variety $\V$
of algebras in the sense of universal algebra
(i.e., the class of all sets given with a family of operations
of specified arities, satisfying a specified set of identities
\cite[Chapter~8]{245}),
and suppose that we are interested in endomaps $m$ of the
underlying sets of objects of $\V$ which satisfy a
certain set of identities in the operations of $\V$ and
the inputs and outputs of $m.$
(E.g., if the variety is $\Rg_K$
and the maps are the derivations, the
identities are (\ref{d.dr+s}), (\ref{d.dcr}) and (\ref{d.drs}).)
For every $A\in\V,$ let $M(A)$ denote the set of all
such maps, and let us call these the $\!M\!$-maps of $A.$
Then we may define an {\em extended inner} $\!M\!$-map of $A$
as a way of associating to each $f:A\to B$ in $(A\downarrow \V)$
an $m_f\in M(B),$ so as to satisfy the analog of~(\ref{d.dfunct}).
If we look at the action of such an extended inner $\!M\!$-map
on a generic element, namely, the element $x\in A\langle x\rangle,$
and call its image $w(x)\in A\langle x\rangle,$
we again find that $w(x)$
determines the whole extended inner $\!M\!$-map;
so we can study such maps by examining such elements.
(The same observations apply, with obvious modifications, if one is
interested in associating to each $f:A\to B$ an indexed {\em family}
of operations on $B,$ each of a specified {\em arity}, satisfying a set
of identities relating them with each other and the operations of $B.$
Then each operation of arity $n$ in the family would have a generic
instance in $A\langle x_1,\dots,x_n\rangle.$
We shall say a little more about this in \S\ref{S.concl},
but will stick to the case of a single unary operation here.)
Formalizing, we have

\begin{definition}\label{D.M}
Suppose we are given a variety $\V$ of algebras in the sense
of universal algebra, and a class of
set-maps $m$ of members $A$ of $\V$ into themselves, which consists of
all set-maps $A\to A$ that satisfy a certain family of identities in
the operations of $\V,$ and which we call ``$M\!$-maps''.
Then an {\em extended inner} $\!M\!$-map of an object $A$
of $\V$ will mean a function associating to every
object $f:A\to B$ of $(A\downarrow \V)$
an $\!M\!$-map $m_f$ of $B,$ such that
for every morphism $h:B_1\to B_2$ in $(A\downarrow\V),$ one has
\begin{equation}\label{d.mfunct}
m_{f_2}\,h\ =\ h\,m_{f_1}.
\end{equation}
\end{definition}

Clearly the concept of an {\em inner endomorphism} of an object of a
general category $\fb{C}$ given by Definition~\ref{D.inner}, when
restricted to the case where $\fb{C}$ is a variety $\V$ of algebras,
agrees with the above definition.
(The inner {\em auto\/}morphisms of an object of a variety are
then the inner endomorphisms $(\beta_f)$ such that
all $\beta_f$ are invertible.)
On the other hand,
systems of maps as in the hypothesis of Theorem~\ref{T.d_inner} can
now, as desired, be described as the {\em extended inner derivations}
of our associative algebra $R.$
In the next section we shall similarly consider extended inner
derivations of Lie algebras.
\vspace{.5em}

Digression:  Having talked about several versions of the concept of a
derivation, let me for completeness recall two more,
though I will not discuss ``extended inner'' versions of these.

The most general of the versions mentioned above, that of
an $\!(h,h')\!$-derivation, is generalized further by the concept of
a derivation from a $\!K\!$-algebra $S$ to an $\!(S,S)\!$-bimodule $B,$
formally defined by our original formulas (\ref{d.dr+s}),
(\ref{d.dcr}) and~(\ref{d.drs}).
In the last of these formulas, and in the analog of~(\ref{d.inner_der})
defining the concept of an {\em inner} derivation $S\to B,$
the ``multiplication'' on the right-hand
sides of these equations is taken to be that of the bimodule structure.
(Thus,~(\ref{d.drs_gh}) and~(\ref{d.inner_der_gh}) are the
cases of~(\ref{d.drs}) and~(\ref{d.inner_der})
where $S_2$ is made an $\!(S_1,S_1)\!$-bimodule
by letting $S_1$ act on the left via the images of its elements
under $h,$ and on the right via their images under $h'.)$
Such a derivation is equivalent to a homomorphism $S\to S\oplus B,$
where $S\oplus B$ is made a $\!K\!$-algebra under a multiplication
based on the multiplication of $S,$ the bimodule structure of
$B,$ and the trivial internal multiplication of $B.$
Each {\em inner} derivation $S\to B$ corresponds to
conjugation by a unit of the form $1+b$ $(b\in B).$

Finally,
in group theory, one sometimes speaks of a left or right derivation
$d$ from a group $G$ to a group $N$ on which $G$ acts by automorphisms.
If we write the action of $G$ on $N$ as left superscripts in
the case of left derivations, and as right superscripts in the case
of right derivations (requiring it to be a left action in
the former case and right action in the latter), and denote
the group operation of $N$ by ``$\cdot$'' (to avoid confusion as to
which elements such superscripts are attached to), then the identities
characterizing these two sorts of maps are
\begin{equation}\label{d.left_rt_der}
d(a\,b)\ =\ d(a)\cdot{^a d(b)},\qquad\mbox{respectively,}\qquad
d(a\,b)\ =\ d(a)^b\cdot d(b).
\end{equation}

In the special case where $N$ is abelian, this concept of derivation
can be reduced to the preceding one.
Indeed, note that a left or right action of $G$
on $N$ is equivalent to a structure on $N$ of left or right
module over the group ring $\mathbb{Z}\,G.$
To supply an action on the other side, and so make $N$
a $\!(\mathbb{Z}\,G,\,\mathbb{Z}\,G)\!$-bimodule,
we map $\mathbb{Z}\,G$ to $\mathbb{Z}$ by the augmentation map,
then use the unique action of $\mathbb{Z}$ on any abelian group.
A derivation $G\to N$ in the sense of~(\ref{d.left_rt_der})
(written now with ``$+$'' instead of ``$\cdot$'') is then a derivation
from the ring $R=\mathbb{Z}\,G$ to its bimodule~$N.$
\vspace{.5em}

Returning to ring theory, let me note yet another way
``inner'' has been used, possibly related to that of this note.
Automorphisms and derivations of a $\!K\!$-algebra $R$ are
{\em actions} on $R$ of certain Hopf algebras, and students of Hopf
algebras have defined what it means for an action of an arbitrary Hopf
algebra on an algebra to be inner \cite{hopf} \cite{SM} \cite{sweedler}.
I am out of my depth in this situation, and do not know
how close that concept is to the concepts of
inner automorphisms and derivations defined here; but
it appears to me that it would be difficult to
embrace under the action of a single Hopf algebra
(or bialgebra) the class of constructions~(\ref{d.S>M_n>S})
with $n$ ranging over all positive integers.

As noted in \cite{hopf}, another case of an inner action of a Hopf
algebra on an algebra $R$ gives us a concept of
an inner {\em grading} of $R$ by a given group or monoid.
It would be interesting to explore the concept of
an ``extended inner grading''.

\section{Inner derivations and inner endomorphisms of Lie algebras.}\label{S.Lie}

Let $K$ be a commutative ring, and $\fb{Lie}_K$ the variety
of Lie algebras over $K.$
Derivations $d:L\to L$
are defined for Lie algebras as for associative algebras,
with~(\ref{d.dr+s}) and~(\ref{d.dcr}) unchanged, and with Lie brackets
replacing multiplication in the analog of~(\ref{d.drs}):
\begin{equation}\label{d.d[rs]}
d([r,\,s])\ =\ [d(r),\,s]\ +\ [r,\,d(s)].
\end{equation}

The derivations of any Lie algebra $L$ or associative algebra $R$
themselves form a Lie algebra under {\em commutator brackets}.
For $L$ a Lie algebra,
there is a natural homomorphism from $L$ to its Lie algebra of
derivations, called the {\em adjoint map}, taking
each $s\in L$ to the derivation $\r{ad}_s: L\to L$ given by
\begin{equation}\label{d.ad}
\r{ad}_s(t)\ =\ [s,\,t]\quad (t\in L).
\end{equation}
For each $s,$ $\r{ad}_s$ is called
the {\em inner derivation} of $L$ determined by $s.$

Clearly, each $s\in L$ induces in this way an {\em extended inner
derivation} of $L$ in the sense of Definition~\ref{D.M}.
To investigate whether these are the only extended
inner derivations, let $B:\Rg_K\to\fb{Lie}_K$ be the functor
that sends each associative $\!K\!$-algebra $R$ to the Lie
algebra having the same underlying
$\!K\!$-module as $R,$ and Lie brackets given by commutator brackets,
\begin{equation}\label{d.B}
[s,\,t]\ =\ s\,t-t\,s.
\end{equation}
We recall that the functor $B$ has a left adjoint,
the {\em universal enveloping algebra} functor $E: \fb{Lie}_K\to\Rg_K.$
(The Poincar\'e-Birkhoff-Witt
Theorem tells us, inter alia, that if $L$ is a Lie algebra over
a field, then the natural map $L\to B(E(L))$ is an embedding.)

Now suppose $(d_f)$ is an extended inner derivation of $L.$
The adjointness relation between $B$ and $E$ tells us that
for $S$ an associative $\!K\!$-algebra, homomorphisms $E(L)\to S$
as associative algebras correspond to Lie homomorphisms $L\to B(S);$
hence if we apply our extended inner derivation to the
latter homomorphisms, we get for every object
$f:E(L)\to S$ of $(E(L)\downarrow\Rg_K)$ a derivation of the Lie
algebra $B(S),$ in a functorial manner.
The condition of being a derivation of $B(S)$ as a Lie algebra
is weaker than that of
being a derivation of $S$ as an associative algebra, so we can't
apply Theorem~\ref{T.d_inner} directly to this family of derivations.
The family will,
however, by the same arguments as before, be determined by an element
$w(x)\in E(L)\langle x\rangle;$ and will satisfy~(\ref{d.dr+s})
and~(\ref{d.dcr}), which we have seen are equivalent to $w(x)$ having
the form $\sum a_i\,x\,b_i,$ with $a_i,\,b_i$ now taken from $E(L).$
The difference between the associative case and the Lie case
rears its head in the equation saying that our
induced maps satisfy~(\ref{d.d[rs]}).
This involves commutator brackets in $E(L)\langle x_0, x_1\rangle$
in place of its associative multiplication;
thus, instead of~(\ref{d.dx0x1}) we get
\begin{equation}\begin{minipage}[c]{35pc}\label{d.d[x0x1]}
$\sum_1^n a_i(x_0\,x_1-x_1\,x_0)b_i\\[.4em]
\hspace*{1in} =\ %
((\sum_1^n a_i\,x_0\,b_i)\,x_1\,-\,x_1\,(\sum_1^n a_i\,x_0\,b_i))\ +
\ (x_0\,(\sum_1^n a_i\,x_1\,b_i)\,-\,(\sum_1^n a_i\,x_1\,b_i)\,x_0).$
\end{minipage}\end{equation}
The added complexity is illusory, however!
Writing $R$ for $E(L),$ note that the terms
of~(\ref{d.d[x0x1]}) lie in the direct
sum of two components $R\,x_0\,R\,x_1\,R\,\oplus\,R\,x_1\,R\,x_0\,R
\,\subseteq\,R\langle x_0, x_1\rangle.$
If we project~(\ref{d.d[x0x1]})
onto the first of these, we get precisely~(\ref{d.dx0x1}), and we can
repeat the computations that led us to Theorem~\ref{T.d_inner}.
A left inverse to the canonical map $K\to E(L),$ as needed for
the proof of that theorem, is supplied by the algebra
homomorphism
\begin{equation}\label{d.E->K}
E(L)\ \to\ K
\end{equation}
that we get on applying $E$ to the trivial map $L\to\{0\}.$

What those computations now tell us is that there is a $b\in E(L)$
such that $w(x)=x\,b-b\,x,$ so that for an object $f:E(L)\to S$
of $(E(L)\downarrow\Rg_k),$ the induced derivation on $B(S)$ is given
by the operation of commutator bracket with $f(b).$
(This shows, incidentally, that that derivation of the Lie algebra
$B(S)$ is in fact a derivation of the associative algebra $S.)$
If we now assume that $K$ is a field, so that every Lie algebra
$M$ can be identified with its image in $E(M),$ we see that
given any Lie algebra homomorphism $f:L\to M,$
the resulting derivation $d_f:M\to M$ can be described within $E(M)$ as
commutator brackets with $f(b).$
(Here we are using the fact
that by the functoriality of our extended inner derivation,
its behavior on $M$ is the restriction of its behavior on $B(E(M)).)$
Also, since elements of $K$ induce the zero derivation, we can assume
without loss of generality that the constant term of $b$ (its
image under~(\ref{d.E->K})) is zero.

This reduces our problem to the question: what elements $b\in E(L)$
with constant term $0$
have the property that for every $f:L\to M,$ the operation of
commutator brackets with
the image of $b$ in $E(M)$ carries $M\subseteq E(M)$ into itself?
Equivalently, what elements $b$ with constant term
zero have the property that the element
$w(x)=x\,b-b\,x\in E(L)\langle x\rangle$ lies in the Lie subalgebra
of $E(L)\langle x\rangle$ generated by $L$ and $x$?
Clearly, all $b\in L$ have this property.
Are they the only ones?

If the field $K$ has positive characteristic $p,$ the answer is no.
It is known that in this case the $\!p\!$-th power of a derivation
of a Lie or associative algebra is again a derivation, and in
particular, that the $\!p\!$-th power of the inner derivation of
an associative algebra determined by an element $a$ is the
inner derivation determined by $a^p.$
(This, despite the fact that the $\!p\!$-th power map does not
in general respect addition on noncommutative $\!K\!$-algebras.)
For nonzero $a\in L\subseteq E(L),$ the element $a^p\in E(L)$ will
not lie in $L;$ so commutator brackets with such elements give
extended inner derivations of $L$ that do not come from
inner derivations in the traditional sense.

The abovementioned fact about $\!p\!$-th powers
of derivations in characteristic $p$ leads to the concept
of a {\em $\!p\!$-Lie algebra} (or {\em restricted} Lie algebra of
characteristic $p$ \cite[\S V.7]{Lie}): a Lie algebra $L$ over a field
of characteristic $p$ with an additional operation of ``formal $p\!$-th
power'', $a\mapsto a^{[p]},$ satisfying appropriate identities.
For this class of objects, one has a ``restricted universal enveloping
algebra'' construction, $E_p(L),$ where relations
are imposed making the formal $\!p\!$-th powers of elements
in the $\!p\!$-Lie algebra coincide with their
ordinary $\!p\!$-th powers in the enveloping algebra.
As we shall note below, this leads to a modified version
in characteristic~$p$ of the question whose unmodified
form we just answered in the negative.

When $K$ has characteristic $0,$
I suspect that a Lie algebra $L$ has no extended
inner derivations other than those induced by elements $b\in L.$
(If there existed general constructions in this case, like
the $\!p\!$-th power operator in the characteristic-$\!p\!$ case,
one would expect the phenomenon to be well-known!)
In any case, we ask

\begin{question}\label{Q.Lie}
If $L$ is a Lie algebra over a field of characteristic $0,$
can commutator brackets with elements $b\in E(L)$ of constant
term zero, other than elements
of $L,$ induce extended inner derivations of $L$?

Same question for $L$ a $\!p\!$-Lie algebra over a field of
characteristic $p>0,$ and $b\in E_p(L).$

These are equivalent to the questions of whether there can exist
in $E(L)$ \textup{(}respectively in $E_p(L))$
elements $b$ of constant term zero
not lying in $L,$ with the property that the element
$w(x)=x\,b-b\,x$ belongs to the Lie subalgebra \textup{(}respectively
the $\!p\!$-Lie subalgebra\textup{)} of $E(L)\langle x\rangle$
\textup{(}respectively $E_p(L)\langle x\rangle)$
generated by $L$ and~$x.$
\end{question}

Just as we have used, above, our analysis of extended inner derivations
on associative algebras in studying extended inner derivations on
Lie algebras, so we can do the same for extended inner endomorphisms
of Lie algebras, again assuming $K$ a field.
If we copy the development of Theorem~\ref{T.end_R-Rg}, taking
$R=E(L),$ and using commutator brackets in place of products,
we can again get from an extended inner endomorphism of a Lie
algebra $L$ an element $w(x)\in E(L)\langle x\rangle,$ which
we find will have the form $\sum a_i\,x\,b_i$
for $a_i,\,b_i\in E(L);$ and the map it induces will
respect commutator brackets on objects of $(E(L)\downarrow\Rg_K).$
That property is equivalent to a formula like~(\ref{d.sumsum}), but
with components in both $R\,x_0\,R\,x_1\,R$ and $R\,x_1\,R\,x_0\,R.$
Again, projection onto the $R\,x_0\,R\,x_1\,R$ component
gives us precisely our old formula, in this case~(\ref{d.sumsum}).
As in \S\ref{S.ring}, this yields~(\ref{d.*dij}).

However, homomorphisms of Lie algebras satisfy no analog of the
condition of sending $1$ to $1;$ so we do not have~(\ref{d.1|->1}),
and cannot deduce~(\ref{d.sum=1}).
What does~(\ref{d.*dij}) tell us without~(\ref{d.sum=1})?
It says that the identity endomorphism of the
free right $\!E(L)\!$-module of dimension $n$ factors
through the free right $\!E(L)\!$-module of dimension~$1.$

Now $E(L)$ admits a homomorphism to the field $K,$
namely~(\ref{d.E->K}),
so such a factorization of maps of free modules can only exist if
such a factorization exists for modules over $K,$ i.e., if $n\leq 1.$
If $n=0$ then $w(x)=0,$ and in contrast to the
case of unital associative rings, this indeed corresponds
to an inner endomorphism of $L$ in $\fb{Lie}_K.$
If $n=1,$ then~(\ref{d.*dij}) becomes $b_1 a_1=1.$
From the fact that the $\!K\!$-algebra $E(L)$ has a filtration
whose associated graded ring is a polynomial ring over $K,$
it follows that, like a polynomial ring, it has no
$\!1\!$-sided invertible elements other than the nonzero elements
of $K;$ so $a_1,\,b_1\in K,$ and we conclude that $w(x)=x.$
Hence,

\begin{theorem}\label{T.end_L}
If $L$ is a Lie algebra over a field $K,$ then its only
extended inner endomorphisms are the zero endomorphism
and the identity automorphism.\qed
\end{theorem}

The above result, even in the characteristic-$\!p\!$ case,
concerns ordinary Lie algebras, not $\!p\!$-Lie algebras.
If $K$ is a field of characteristic $p>0,$ and $L$ a $\!p\!$-Lie
algebra over $K,$ we can begin the analysis of extended inner
endomorphisms of $L$ as above, with
$E_p(L)$ in place of $E(L),$ and go through much the
same argument, using as before the fact that $w(x)$ respects Lie
brackets, and conclude that every extended inner
endomorphism is either zero, or induced by an element $w(x)=a\,x\,b$
for $a,\,b\in E_p(L)$ satisfying $b\,a=1.$
(Note that this automatically implies that $w(x^p)=w(x)^p.)$
But we can no longer say that the relation $ba=1$
implies that $a,\,b\in K.$
For example, if $u$ is an element of $L$ such that $u^{[p]}=0,$ then
in $E_p(L)$ we have $u^p=0,$ so $1-u$ is a nonscalar invertible
element.
Hence we ask

\begin{question}\label{Q.pLie}
Can a $\!p\!$-Lie algebra $L$ over a field
$K$ have a nonzero non-identity extended inner endomorphism?

Equivalently, can $E_p(L)$ have elements $a,\,b,$ not in $K,$
satisfying $b\,a=1,$ and such
that in $E(L)\langle x\rangle,$ the element $w(x)=a\,x\,b$
lies in the $\!p\!$-Lie subalgebra generated by $L$ and $x$?
\end{question}

The first part of the above question can be divided in two:
Can such an $L$ have an extended inner {\em automorphism}
that is not the identity?
and can it have a nonzero extended inner endomorphism
that is not an extended inner automorphism?
The latter possibility can in turn be divided in two:
There might be an invertible element, conjugation by
which carries the $\!p\!$-Lie subalgebra generated by $L$ and $x$
into, but not onto, itself, or the extended inner endomorphism
might arise from elements $a,\ b$ such that $ba=1$ but $ab\neq 1;$
I do not know whether an enveloping algebra $E_p(L)$
can contain one-sided but not two-sided invertible elements.

However, we can again say that a nonzero extended
inner endomorphism is everywhere one-to-one.
For if our $w(x)=a\,x\,b,$ and if on mapping $x$ to some element
$u\in L'$ under a map $L\langle x\rangle\to L',$ we get
$a\,u\,b=0,$ then by multiplying this equation on the left by $b$ and
on the right by $a,$ we find that $u=0.$\vspace{.5em}

It is natural to ask
whether the methods we have used to study inner automorphisms, inner
endomorphisms, and inner derivations of associative and Lie algebras
are applicable to other classes of
not-necessarily-associative algebras.
Our results for associative algebras used the descriptions~(\ref{d.Rx})
and~(\ref{d.Rx0x1}) of the free extensions $R\,\langle x\rangle$
and $R\,\langle x_0,\,x_1\rangle$ of an algebra $R;$
and our partial results for Lie algebras
were based on reduction to the associative case.
For most varieties of $\!K\!$-algebras, the
descriptions of the universal one- and two-element
extensions of an algebra are not so simple.
I have not examined what can be proved in such cases.

\section{Co-inner endomorphisms.}\label{S.co-inner}

If $A$ is an object of a category $\fb{C},$
there is a construction dual to that of $(A\downarrow\fb{C}),$
namely $(\fb{C}\downarrow A),$ the category whose objects are
objects of $\fb{C}$ given with maps to $A,$ and morphisms
making commuting triangles with those maps.
Thus, we may dualize Definition~\ref{D.inner}, and define an
{\em extended co-inner} endomorphism of an object $A$ of
$\fb{C}$ to mean an endomorphism $E$ of the forgetful
functor $(\fb{C}\downarrow A)\to\fb{C},$ and a {\em co-inner}
endomorphism of $A$ itself to mean the value of such a morphism on $A.$

I don't know of important
naturally occurring examples, and I suspect that if
the concept turns out to be useful, it will be
so mainly in areas other than
algebra; but let us make a few observations on the algebra case.

Let $\bf{V}$ be a variety of algebras in the sense of universal algebra.
We begin with the weaker concept
of an extended co-inner {\em set-map} of $A;$
that is, an endomorphism $E$ of the composite of forgetful functors
\begin{equation}\label{d.>V>Set}
(\V\downarrow A)\ \longrightarrow\ \V
\ \longrightarrow\ \fb{Set}.
\end{equation}

To analyze such a mapping, let us, for each
$a\in A,$ consider the object of $(\V\downarrow A)$ given
by the homomorphism from the free $\!\V\!$-object on one generator,
$\langle x\rangle_\V,$ to $A,$ that takes $x$ to $a.$
If we apply our co-inner set-map $E$ to this homomorphism,
we get a set-map $\langle x\rangle_\V\to \langle x\rangle_\V;$
this will take $x$ to some element
$w_a(x)\in\langle x\rangle_\V;$ thus we get a family of
such elements $w_a(x)\in\langle x\rangle_\V,$ one for each $a\in A.$
We see that this family will determine $E;$
namely, for every object $f:B\to A$
of $(\V\downarrow A),$ and every element $b\in B,$
$E_f$ will take $b$ to $w_{f(b)}(b).$
Clearly, any $\!A\!$-tuple $(w_a(x))_{a\in A}$ of elements
of $\langle x\rangle_\V$ yields such an extended
``co-inner set map'' $E.$
(Remark: though there is an added complexity relative to the case of an
extended inner endomorphism of an algebra, in that we now have a
family of elements $w_a(x)$ rather than a single element $w(x),$
there is a corresponding decrease in complexity, in that these
lie in $\langle x\rangle,$ rather than~$A\langle x\rangle.)$

For most $\V,$ few extended co-inner
set maps will give endomorphisms of the algebras $B.$
One way to get examples which do so is to take all $w_a(x)$ the same,
with value giving what we called
in \S\ref{S.other} an ``absolute endomorphism of $\V$''.
E.g., for $\V=\fb{Ab}$ and $A$ any abelian
group, we may take all $w_a(x)$ equal to $n\,x$ for a fixed $n.$
More generally, for $\V$ the variety of modules over
a ring $R$ and $A$ any such module, we may take all
$w_a(x)$ equal to $c\,x$ for a fixed element $c$ of the center of~$R.$

However, here is a class of cases in which not all
co-inner endomorphisms are based on absolute endomorphisms.

\begin{theorem}\label{T.Set-G}
Let $G$ be a group, let $\fb{Set}_G$ be the variety of
right $\!G\!$-sets, let $A$ be an object of this variety,
and let $S$ be a set of representatives of the
orbits of $A$ under $G.$

Then every extended co-inner endomorphism $E$ of $A$ in
$\fb{Set}_G$ is an extended co-inner {\em automorphism},
and may be constructed by choosing,
for each $s\in S,$ an element $g_s$ of the centralizer
in $G$ of the stabilizer $G_s$ of $s,$
and for each $s\,h\in A$ $(s\in S,\ h\in G),$
letting $w_{sh}(x)=x\,h^{-1}g_s\,h.$

The extended co-inner endomorphisms of $A$ thus form a
group, isomorphic to the direct product, over $s\in S,$ of the
centralizers of the stabilizer subgroups $G_s.$
\end{theorem}

\begin{proof}
To get an extended co-inner endomorphism of $A,$ we must choose
for each $a\in A$ an element $w_a(x)$ of the free
$\!G\!$-set on one generator, which we will denote $x\,G,$
in a way that makes the resulting extended inner set-map
consist of morphisms of $\!G\!$-sets.
By the structure of $x\,G,$ we see that for each
$a\in A$ we have $w_a(x)=x\,g_a$ for a unique $g_a\in G.$

The condition for these maps to induce morphisms of $\!G\!$-sets
is that for every $a\in A$ and $h\in G,$
$w_a(x)\,h=w_{ah}(x\,h),$ in other words
\begin{equation}\label{d.Exh}
x\,g_a\,h\ =\ (x\,h)\,g_{ah}.
\end{equation}
The above equality is equivalent to $g_a\,h=h\,g_{ah},$ or
solving for $g_{ah},$
\begin{equation}\label{d.gh}
g_{ah}\ =\ h^{-1}g_a\,h\quad (a\in A,\ h\in G).
\end{equation}

If $h$ lies in the stabilizer
subgroup $G_a,$ we have $a\,h=a,$ so $g_{ah}=g_a,$ so
in this case~(\ref{d.gh}) says that $g_a$ commutes with $h.$
Hence $g_a$ lies in the centralizer of $G_a.$
For general $h,$~(\ref{d.gh}) allows us to compute $g_{ah}$ from $g_a,$
hence, the system of elements $g_a$ will be determined by those such
that $a$ lies in our set of coset representatives $S,$ and
the value at each $s\in S$ will belong to the centralizer of~$G_s.$

For elements $g_s$ so chosen, it is now easy to verify that we indeed
get an extended co-inner endomorphism of $A.$
The resulting endomorphisms are clearly invertible, and the description
of the group they form is immediate.
\end{proof}

\section{Concluding remarks.}\label{S.concl}

The tools used in \S\S\ref{S.group}-\ref{S.Lie} above are not new
from the point of view of category-theoretic universal algebra.

If we consider the general context of ``$\!M\!$-maps'' as in
\S\ref{S.deriv_def}, and then pass to the still more general
context, sketched parenthetically there, of a family of additional
operations on the underlying set of an object of $\V,$
of various arities, subject to some set of identities, we see that
this constitutes a structure of algebra in a
variety $\fb{W}$ whose
operations and identities include those of $\V.$
An ``extended inner system'' of such operations on an object
$A$ of $\V$ then means a factorization of the forgetful functor
$(A\downarrow\V)\to\V$ through the forgetful functor $\fb{W}\to\V.$
From the point of view of the theory of representable algebra-valued
functors (\cite{Freyd}, \cite[Chapter~9]{245},
\cite[Chapters I-II]{coalg}), this corresponds to starting with the
representing object for the former forgetful functor,
namely, $A\langle x\rangle$ with the canonical system of
co-operations that make it
a co-$\!\V\!$-object of the variety $(A\downarrow\V),$ and enhancing
that co-$\!\V\!$-structure in an arbitrary
way to a co-$\!\fb{W}\!$-structure; i.e.,
supplying additional co-operations
which co-satisfy the identities of $\fb{W}.$
These co-operations will be determined by their actions
on the element $x,$ so by studying the images of $x$ under them,
one may attempt to determine the
form that the additional co-operations can take.

Thus, what we have been doing falls under the general study of
representable functors and the coalgebras that represent them.
I consider the contribution of this note not to lie in the
maximum generality to which the concepts could be pushed (which
comes to that existing general theory),
but, inversely, in the focus on a specific class
of such problems: those where an added unary operation constitutes
a type of additional structure on the objects in question
that is already of interest, e.g., an endomorphism, or a derivation.
We have gotten exact descriptions of the possibilities for this
structure in several such cases, and
shown the technique that can be applied to further cases.

Of course, if this note leads some readers to an interest in the
general theory of coalgebras and representable functors among
varieties of algebras
\cite{Freyd}, \cite[Chapter~9]{245}, \cite{coalg}, I will
be all the more pleased.

\section{Appendix: Inner endomorphisms of associative algebras are one-to-one.}\label{S.1-1}

We noted in the second paragraph after Theorem~\ref{T.end_R-Rg}
that the one-one-ness of every inner endomorphism of
an associative unital algebra $R$
over a field $K,$ which follows from that theorem, also has
an elementary proof, using the fact that $R$
can be embedded in a simple $\!K\!$-algebra.
We prove below a different embedding result, from which
we deduce, more generally,
the one-one-ness of all inner endomorphisms of associative
unital algebras over arbitrary $K.$

Below, $K$ is, as usual, a commutative associative unital ring,
and $\otimes$ denotes $\otimes_K.$
$\!K\!$-algebras are here understood to be associative and unital.

\begin{lemma}\label{L.meet_ctr}
Every $\!K\!$-algebra $R$ admits an embedding $f:R\to S$ in
a $\!K\!$-algebra $S$ with the property that for every
nonzero $r\in R,$ the ideal $S\,f(r)\,S$ contains a nonzero
element of the center of $S.$
\end{lemma}

\begin{proof}
Given $R,$ first form the $\!K\!$-module $R\otimes R,$ and
note that the two maps $R\to R\otimes R$ given by
$r\mapsto r\otimes 1$ and $r\mapsto 1\otimes r$ are one-to-one,
since the map $R\otimes R\to R$ induced by the internal multiplication
of $R$ gives a left inverse to each of them.

Now regard $R\otimes R$ as a $\!K\!$-algebra in the usual way,
i.e., so that $(r_1\otimes r_2)\cdot (r'_1\otimes r'_2)=
(r_1r'_1\otimes r_2r'_2).$
By the above observation, $r\mapsto r\otimes 1$ and
$r\mapsto 1\otimes r$ are embeddings of $\!K\!$-algebras.
Note that their images centralize one another, and that
the map $\theta:R\otimes R\to R\otimes R$ defined by
$\theta(r_1\otimes r_2)=r_2\otimes r_1$ is an automorphism of
$R\otimes R.$
Using this automorphism, let us form the twisted polynomial
algebra $(R\otimes R)[t;\theta];$
i.e., adjoin to $R\otimes R$ an indeterminate $t$ satisfying
\begin{equation}\label{d.trs}
t\,(r_1\otimes r_2)\ =\ (r_2\otimes r_1)\,t\quad
\r{for}\ \ r_1,\,r_2\in R.
\end{equation}

Within $(R\otimes R)[t;\theta],$ we now
take the subalgebra of elements whose
constant terms lie in $R\otimes 1,$ and let $S$
be the quotient of this subalgebra by the ideal of all elements in
which $t$ appears with exponent $>2.$
Thus, as a $\!K\!$-module,
\begin{equation}\label{d.S=}
S\ =
\ (R\otimes 1)\ \oplus\ (R\otimes R)\,t\ \oplus\ (R\otimes R)\,t^2.
\end{equation}
We now define our algebra embedding $f:R\to S$ by
\begin{equation}\label{d.f:R>S}
f(r)\ =\ r\otimes 1.
\end{equation}
For every nonzero $r\in R,$ the ideal $S\,f(r)\,S$ contains the element
\begin{equation}\label{d.tfrt}
t\,f(r)\,t\ =\ t\,(r\otimes 1)\,t\ =\ (1\otimes r)\,t^2,
\end{equation}
which we see from the right-hand side of the above equation is nonzero.
Because this element involves $t$ to the second power, it
annihilates on both sides the summands of~(\ref{d.S=}) involving $t.$
It also centralizes the summand $R\otimes 1,$ since
the factors $1\otimes r$ and $t^2$ both do so.
So~(\ref{d.tfrt}) gives the desired nonzero central element.
\end{proof}

To make use of this result, recall that in our category of
$\!K\!$-algebras, an endomorphism of an object by definition
fixes the unit, and that in \S\ref{S.ring}
we translated this to the condition~(\ref{d.sum=1})
on extended inner endomorphisms.
The proof of the next result shows that in this respect, extended inner
endomorphisms cannot tell the difference between the unit
and other $\!R\!$-centralizing elements.

\begin{lemma}\label{L.fix_ctr}
If $R$ is a $\!K\!$-algebra, and $(\beta_f)$ an
extended inner endomorphism of $R,$ then for every homomorphism
$f:R\to S$ of $\!K\!$-algebras, the endomorphism
$\beta_f$ of $S$ fixes all elements of $S$ that centralize
$f(R),$ hence, in particular, all elements of the center of $S.$
\end{lemma}

\begin{proof}
By abuse of notation, let us use the same symbols for
elements of $R$ and their images in $S.$
If $c\in S$ centralizes $R,$ then applying~(\ref{d.wx=sum}) to $c,$
and commuting $c$ past the coefficients $b_i\in R,$ we get
$\beta_f(c)\ =\ \sum_1^n a_i\,b_i\,c,$
which by~(\ref{d.sum=1}) simplifies to~$c.$
\end{proof}

We can now prove

\begin{proposition}\label{P.1-1}
Every inner endomorphism $\alpha$ of an associative unital
$\!K\!$-algebra $R$ is one-to-one.
\end{proposition}

\begin{proof}
Given $R,$ take an embedding $R\to S$ as in Lemma~\ref{L.meet_ctr}.
Thus for any nonzero $r\in R,$ $S\,f(r)\,S$ contains
a nonzero central element $c.$
By Lemma~\ref{L.fix_ctr}, $\beta_f(c)=c.$
So
\begin{equation}\label{d.bfc}
0\ \neq\ c\ =\ \beta_f(c)\ \in\ \beta_f(S\,f(r)\,S)\ \subseteq
\ S\,\beta_f(f(r))\,S\ =\ S\,f(\alpha(r))\,S,
\end{equation}
so $\alpha(r)\neq 0.$
\end{proof}

Incidentally, in Lemma~\ref{L.meet_ctr},
we made our construction satisfy the strong conclusion
that $S\,f(r)\,S$ have nonzero intersection
with the center of $S,$ since that seemed
of independent interest; but for the proof of Proposition~\ref{P.1-1},
it would have sufficed that $S\,f(r)\,S$ have nonzero intersection
with the centralizer of $f(R).$
This could have been achieved by the simpler construction
\begin{equation}\label{d.S=2}
S\ =\ (R\otimes R)[t;\theta],
\end{equation}
with $f:R\to S$ again defined by $f(r)=r\otimes 1.$
Indeed, $t\,f(r)\,t=(1\otimes r)\,t^2$ clearly still centralizes
$f(R)=R\otimes 1.$


\noindent
\end{document}